\documentclass[10pt, oneside]{amsart}
 \usepackage[text={5.58in,8.5in},centering,letterpaper,dvips]{geometry}
\usepackage{graphicx}
\usepackage{amsfonts}
\usepackage{epsf}
\usepackage{amssymb}
\usepackage{amsmath}
\usepackage{amscd}
\usepackage{tikz}
\usepackage{pdfpages}
\usepackage{fancyhdr}
\usepackage{setspace}
\usepackage{hyperref}
\usepackage[all]{xy}
\usetikzlibrary{matrix}
\usepackage{verbatim}
\usepackage{enumerate}

\theoremstyle{theorem}
\newtheorem{theorem}{Theorem}[section]
\newtheorem{proposition}[theorem]{Proposition}

\newtheorem{question}[theorem]{Question}
\newtheorem{questions}[theorem]{Questions}
\newtheorem{corollary}[theorem]{Corollary}
\newtheorem{conjecture}[theorem]{Conjecture}

\theoremstyle{definition}

\newtheorem{remark}[theorem]{Remark}

\newcommand{\Z}{\mathbb{Z}}
\newcommand{\C}{\mathbb{C}}
\newcommand{\N}{\mathbb{N}}
\newcommand{\Q}{\mathbb{Q}}
\newcommand{\R}{\mathbb{R}}

\newcommand{\CP}{\mathbb{CP}}

\newcommand{\id}{\text{id}}

\newcommand{\Dd}{\mathcal D}
\newcommand{\Hh}{\mathcal H}
\newcommand{\Kk}{\mathcal K}
\newcommand{\Pp}{\mathcal P}
\newcommand{\Tt}{\mathcal T}
\newcommand{\Ss}{\mathcal S}

\makeatletter
\newtheorem*{rep@theorem}{\rep@title}
\newcommand{\newreptheorem}[2]{%
\newenvironment{rep#1}[1]{%
 \def\rep@title{#2 \ref{##1}}%
 \begin{rep@theorem}}%
 {\end{rep@theorem}}}
\makeatother

\newreptheorem{theorem}{Theorem}
\newreptheorem{lemma}{Lemma}
\newreptheorem{question}{Question}
\newreptheorem{corollary}{Corollary}
\newreptheorem{proposition}{Proposition}

\begin{document}

\rhead{\thepage}
\lhead{\author}
\thispagestyle{empty}

%\tableofcontents
%\listoffigures

\raggedbottom
\pagenumbering{arabic}
\setcounter{section}{0}

%%%%%%%%%%%%%%%%%%%%%%%%%%%%%%%%%%%%%%%%%%%%%%%%%%%%%%%%
%%%%%%%%%%%%%%%%%%%%%%%%%%%%%%%%%%%%%%%%%%%%%%%%%%%%%%%%
%%%%%%%%%%%%%%%%%%%%%%%%%%%%%%%%%%%%%%%%%%%%%%%%%%%%%%%%

\title{Trisections and spun 4--manifolds}
%\date{\today}

\author{Jeffrey Meier}
\address{Department of Mathematics, University of Georgia, Athens, GA 30606}
\email{jeffrey.meier@uga.edu}
\urladdr{jeffreymeier.org} 

\begin{abstract}
	We study trisections of 4--manifolds obtained by spinning and twist-spinning 3--manifolds, and we show that, given a (suitable) Heegaard diagram for the 3--manifold, one can perform simple local modifications to obtain a trisection diagram for the 4--manifold.  We also show that this local modification can be used to convert a (suitable) doubly-pointed Heegaard diagram for a 3--manifold/knot pair into a doubly-pointed trisection diagram for the 4--manifold/2--knot pair resulting from the twist-spinning operation.
	
	This technique offers a rich list of new manifolds that admit trisection diagrams that are amenable to study.  We formulate a conjecture about 4--manifolds with trisection genus three and provide some supporting evidence.
\end{abstract}

\maketitle

%%%%%%%%%%%%%%%%%%%%%%%%%%%%%%%%%%%%%%%%%%%%%%%%%%%%%%%%
%%%%%%%%%%%%%%%%%%%%%%%%%%%%%%%%%%%%%%%%%%%%%%%%%%%%%%%%
\section{Outline}\label{sec:intro}
%%%%%%%%%%%%%%%%%%%%%%%%%%%%%%%%%%%%%%%%%%%%%%%%%%%%%%%%
%%%%%%%%%%%%%%%%%%%%%%%%%%%%%%%%%%%%%%%%%%%%%%%%%%%%%%%%

The theory of trisections was introduced by Gay and Kirby as a novel way of studying the smooth topology of 4--manifolds~\cite{Gay-Kirby_Trisecting_2016}.  Since then, the theory has developed in a number of directions: 
Extensions of the theory to the settings of manifolds with boundary~\cite{Castro_Thesis_,Castro_Trisecting_2017,Castro-Gay-Pinzon-Caicedo_Diagrams_2016}, 
knotted surfaces~\cite{Meier-Zupan_Bridge_2015}, 
algebraic objects~\cite{Abrams-Gay-Kirby_Group_2016}, 
and higher dimensional manifolds~\cite{Rubinstein-Tillmann_Multisections_2016} have been established;
programs offering connections with singularity theory~\cite{Castro-Ozbagci_Trisections_2017,Gay-Kirby_Trisecting_2016,Gay_Trisections_2016,Gay_Functions_2017}, 
and Dehn surgery~\cite{Meier-Schirmer-Zupan_Classification_2016,Meier-Zupan_Characterizing_}, 
have been initiated; 
some classification results have been obtained~\cite{Meier-Schirmer-Zupan_Classification_2016,Meier-Zupan_Genus-two_2017};
 interpretations of constructions and cut-and-paste operation have been explored~\cite{Gay-Meier_Trisections_};
 and new invariants have been proposed~\cite{Gukov_Trisecting_2017,Islambouli_Comparing_2017}.  The purpose of this note is two-fold: motivate an extension of the classification program and generate a rich set of examples of manifolds with trisection diagrams that are simple enough to be amenable to study.

Manifolds with trisection genus at most one are easy to classify~\cite{Gay-Kirby_Trisecting_2016}. In~\cite{Meier-Zupan_Genus-two_2017}, it was shown that $S^2\times S^2$ is the unique irreducible\footnote{We call a 4--manifold $X$ \emph{irreducible} if each summand of any connected sum decomposition of $X$ is either $X$ or a homotopy 4--sphere.} manifold with trisection genus two, and it was asked to what extent it is possible to enumerate manifolds with trisection genus $g$ for low values of $g$.  To this end, we offer the following conjecture.

\begin{conjecture}\label{conj:3}
	Every irreducible 4--manifold with trisection genus three is either the spin of a lens space, or a Gluck twist on a specific 2--knot in the spin of a lens space.
\end{conjecture}

These manifolds have rich but fairly obfuscated history of study in the literature, which we aim to unify in the discussion below.  Since there is a unique spun lens space for each $p\in\N$ and at most one additional manifold obtained by the specified Gluck twist, this conjecture would give an extremely simple enumeration of manifolds admitting minimal genus $(3,1)$--trisections.
%There are a number of interesting ways to describe such 4--manifolds.  Pao first studied them as admitting effective $T^2$--actions~\cite{Pao_The-topological_1977}, and  Hayano showed that they admit genus one broken Lefschetz fibrations~\cite{Hayano}, which led Baykur and Saeki to recently identify them as admitting $(3,1)$--trisections~\cite{Baykur-Saeki_Simplifying_2017}.
(Note that $(3,2)$--trisections are trivial in a precise sense~\cite{Meier-Schirmer-Zupan_Classification_2016}, while $(3,0)$--trisections are conjecturally trivial in the same sense, so Conjecture~\ref{conj:3} can really the thought of as a conjecture about manifolds with irreducible $(3,1)$--trisections.) At the end of the paper, we present diagrams for the subjects of Conjecture~\ref{conj:3}.

Given a closed, connected, orientable 3--manifold $M$, let $\Ss(M)$ and $\Ss^*(M)$ denote the \emph{spin} and \emph{twisted-spin} of $M$, respectively.  (See Section~\ref{sec:proofs} for precise definitions.)

\begin{theorem}\label{thm:SpunTri}
	Suppose that $M$ admits a genus $k$ Heegaard splitting.  Then $\mathcal S(M)$ and $\mathcal S^*(M)$ admit $(3k,k)$--trisections.
\end{theorem}

An immediate application of this theorem is an explicit description of 4--manifolds admitting minimal genus trisections of arbitrarily large genus.

\begin{corollary}\label{coro:Minimal}
	For every integer $g\geq 3$ and every $1\leq k\leq g-2$, there exist infinitely many distinct 4--manifolds admitting minimal $(g,k)$--trisections.  
\end{corollary} 

A similar corollary has been independently obtained recently by Baykur and Saeki~\cite{Baykur-Saeki_Simplifying_2017}.  Corollary~\ref{coro:Minimal} becomes more interesting in light of our ability to give diagrams for the pertinent trisections.

\begin{theorem}\label{thm:diags}
	Let $(S,\delta,\varepsilon)$ be a genus $g$ Heegaard diagram for a closed 3--manifold $M$ with the property that $H_\varepsilon$ is standardly embedded in $S^3$.  Then the 4--manifolds $\Ss(M)$ and $\Ss^*(M)$ each admit a trisection diagram that is obtained from $(S,\delta,\varepsilon)$ via a local modification at each curve of $\varepsilon$.
\end{theorem}

The local moves are described in Figures~\ref{fig:LocalDiag}~and~\ref{fig:LocalDiag*}. See Section~\ref{sec:proofs} for a more detailed statement of the above theorem.

Finally, we consider what happens when the twist-spinning construction is applied to a 3--manifold/knot pair.  Our main result to this end is that the twisted-spin of a doubly-pointed Heegaard diagram is a doubly-pointed trisection diagram.  This latter object describes not only the trisected 4--manifold, but also a knotted sphere therein. Given a 3-manifold/knot pair $(M,K)$, let $\Ss^n(M,K)$ denote the $n$--twist-spin of $(M,K)$.

\begin{theorem}\label{thm:dpdiags}
	Let $(S,\delta,\varepsilon)$ be a genus $g$ Heegaard diagram for a closed 3--manifold $M$ with the property that $H_\varepsilon$ is standardly embedded in $S^3$.  Let $K$ be a knot in $M$ such that $(S,\delta,\varepsilon,z,w)$ is a doubly-pointed Heegaard diagram for the pair $(M,K)$. Then the pairs $\Ss^n(M,K)$  admit doubly-pointed trisection diagrams that are obtained from $(S,\delta,\varepsilon,z,w)$ via a local modification at each curve of~$\varepsilon$.
\end{theorem}

\subsection*{Organiziation}

Section~\ref{sec:back} presents general background material regarding spinning and twist-spinning, Heegaard splittings and trisections, and doubly-pointed diagrams.  In Section~\ref{sec:proofs}, we give a singularity theoretic proof of Theorem~\ref{thm:SpunTri}, and more geometric proofs of Theorems~\ref{thm:diags} and~\ref{thm:dpdiags}, the former of which also recovers a proof of Theorem~\ref{thm:SpunTri}.  In Section~\ref{sec:exs}, we prove Corollary~\ref{coro:Minimal}, discuss Conjecture~\ref{conj:3}, and give some examples.

\subsection*{Acknowledgements}

This article was inspired by conversations with Alex Zupan, who gave a preliminary sketch of the proof of Theorem~\ref{thm:SpunTri}.  The author is also grateful to R. \.{I}nan\c{c} Baykur for helpful conversations that gave a singularity theory context to the present work and for comments that improved the exposition of the article throughout. This work was supported by NSF grants DMS-1400543 and DMS-1664540

%%%%%%%%%%%%%%%%%%%%%%%%%%%%%%%%%%%%%%%%%%%%%%%%%%%%%%%%%%%%%%%%%%%%%%%%%%%%%%%%%%
%%%%%%%%%%%%%%%%%%%%%%%%%%%%%%%%%%%%%%%%%%%%%%%%%%%%%%%%%%%%%%%%%%%%%%%%%%%%%%%%%%
\section{Background}\label{sec:back}
%%%%%%%%%%%%%%%%%%%%%%%%%%%%%%%%%%%%%%%%%%%%%%%%%%%%%%%%%%%%%%%%%%%%%%%%%%%%%%%%%%
%%%%%%%%%%%%%%%%%%%%%%%%%%%%%%%%%%%%%%%%%%%%%%%%%%%%%%%%%%%%%%%%%%%%%%%%%%%%%%%%%%

%%%%%%%%%%%%%%%%%%%%%%%%%%%%%%%%%%%%%%%%%%%%%%%%%%%%%%%%%%%%%%%%%%%%%%%%%%%%%%%%%%
\subsection{Spun 4--manifolds and 2--knots}\ 
%%%%%%%%%%%%%%%%%%%%%%%%%%%%%%%%%%%%%%%%%%%%%%%%%%%%%%%%%%%%%%%%%%%%%%%%%%%%%%%%%%

We recall the set-up of spun 4--manifolds, as well as some classical results about these spaces. Given a closed, connected 3--manifold $M$, we let $\Ss(M)$ and $\Ss^*(M)$ denote the \emph{spin} and \emph{twisted-spin} of $M$, respectively.  These manifolds are given as follows:
$$\Ss(M) = (M^\circ\times S^1)\cup_\id (S^2\times D^2),$$
and
$$\Ss^*(M) = (M^\circ\times S^1)\cup_\tau (S^2\times D^2),$$
where $\tau$ is the unique self-diffeomorphism of $S^2\times S^1$ not extending over $S^2\times D^2$~\cite{Gluck_The-embedding_1962}.  Adopting coordinates $(h, \phi)$ for $S^2$, where $h\in[-1,1]$ represents distance from the equator and $\phi\in S^1$ is angular displacement from a fixed longitude, this map is given by
$$\tau((h,\phi),\theta) = ((h,\phi+\theta),\theta).$$
In other words, $\tau$ twists $S^2$ through one full rotation as we traverse the $S^1$ direction.  In fact, one could consider gluings using powers of $\tau$, but the resulting manifold will only depend (up to diffeomorphism) on the parity of the power~\cite{Gluck_The-embedding_1962}.

Such spaces were well studied in the 1980s and earlier.  Here, we will summarize some of the more pertinent facts.  We denote diffeomorphism and homotopy-equivalence by $\cong$ and $\simeq$, respectively. It appears that a complete classification of when the spin and twisted-spin of a given 3--manifold are diffeomorphic remains open.  However, we have the following significant progress due to Plotnick.

\begin{theorem}[Plotnick~\cite{Plotnick_Equivariant_1986}]\label{thm:Plotnick}
	Let $M$ be a closed, connected, orientable 3--manifold.
	\begin{enumerate}
		\item If $M$ is aspherical, then $\Ss(M)\not\simeq\Ss^*(M)$.
		\item $\Ss(M)\cong\Ss^*(M)$ if every summand of $M$ is either $S^1\times S^2$ or a spherical 3--manifold with all Sylow subgroups of $\pi_1(M)$ cyclic.
	\end{enumerate}
\end{theorem}

\begin{remark}\label{rmk:3-skeleta}
	Note that $\Ss(M)$ and $\Ss^*(M)$ have identical 3--skeleta.  One way to see this is to notice that both of these manifolds are obtained from $M\times S^1$ by surgering a circle $\ast\times S^1$, with the result only depending on the choice of framing in $\pi_1(SO(3))\cong\Z_2$.  Since the framings can be assumed to agree on a portion of $\ast\times S^1$, it follows that the surgeries differ only in the attaching of a 4--cell.  As a consequence $\pi_1(\Ss(M))\cong\pi_1(\Ss^*(M))$, and it is not hard to argue that this group is simply $\pi_1(M)$.
\end{remark}

By the above remark, $\Ss(L(p,q))$ can be obtained by surgering out $S^1\times\ast$ inside $S^1\times L(p,q)$.  Pao observed that $\Ss(L(p,q))$ can also be obtained by surgering the simple closed curve in $S^1\times S^3$ representing $p\in\Z\cong\pi_1(S^1\times S^3)$~\cite{Pao_The-topological_1977}.  As in Remark~\ref{rmk:3-skeleta}, there are two choices for the framing of such a surgery.  Let $\Ss_p$ and $\Ss_p'$ denote the manifolds obtained from surgery on the winding number $p$ curve in $S^1\times S^3$.  (Note that it follows that $\Ss_p$ and $\Ss_p'$ are related by a Gluck twist on the belt-sphere of this surgery.) Pao proved the following.

\begin{proposition}[Pao~\cite{Pao_The-topological_1977}]\label{prop:Pao}
	\ 
	\begin{enumerate}
		\item $\Ss_p\cong\Ss(L(p,q))$.
		\item $\Ss'_p\cong\Ss_p$ if $p$ is odd and $\Ss'_p\not\simeq\Ss_p$ if $p$ is even.
	\end{enumerate}
\end{proposition}

We remark that it is not clear whether Pao identified $\Ss_p$ as a spun lens space, though it appears that Plotnick made the connection~\cite{Plotnick_Equivariant_1986}. (See also~\cite{Suciu_The-oriented_1988}.)  Moreover, many authors who have studied Pao's manifolds since seem not to have noted the connection with spun lens spaces, instead studying them as manifolds admitting genus one broken Lefschetz fibrations~\cite{Baykur-Kamada_Classification_2015,Baykur-Saeki_Simplifying_2017,Hayano_On-genus-1_2011}.

Combining Theorem~\ref{thm:Plotnick}(2) and Proposition~\ref{prop:Pao}(1), we have the following corollary.

\begin{corollary}\label{coro:lens}
	For all $1\leq q<p$, both $\Ss(L(p,q))$ and $\Ss^*(L(p,q))$ are diffeomorphic to $\Ss_p$.
\end{corollary}

We will let $\Pp = \{\Ss_p\}_{p\in\N}\cup\{\Ss_p'\}_{p\in2\N}$ be the set of Pao's manifolds, and we will refer to the $\Ss_p$ as the \emph{spun lens spaces} and to the $\Ss_p'$ as their \emph{siblings}.  

\begin{remark}
	Note that there are two pertinent 2--knots in the manifold $\Ss_p = \Ss(L(p,q))$.  The first is the core of the $D^2\times S^2$ used in the spinning construction.  Performing a Gluck twist on this 2--knot results in $\Ss^*(L(p,q))$, while surgery yields $S^1\times L(p,q)$.  The second 2--knot has the property that surgery yields $S^1\times S^3$; thus, it cannot be isotopic to the first 2--knot.  Performing a Gluck twist on this latter 2--knot results in the sibling manifold $\Ss'_p$.
\end{remark}

Finally, we extend the definition of twist-spinning to 3--manifold/knot pairs. For a fixed 3--manifold $M$ and a knot $K$ in $M$, let $\Ss^n(M,K)$ denote the \emph{$n$--twist-spin} of the pair $(M,K)$:
$$\Ss^n(M,K) = ((M,K)^\circ\times S^1)\bigcup_{\tau^n}(S^2\times D^2,\{\frak n,\frak s\}\times D^2),$$
where the gluing is via the $n$--fold power of the Gluck twist map defined above.  We write $\Ss^k(M,K) = (\Ss^k(M),\Ss^k(K))$. Since $\tau^2$ extends over $S^2\times D^2$, we have that $\Ss^k(M)$ is either $\Ss(M)$ or $\Ss^*(M)$ (based on whether $k$ is even or odd).  On the other hand, the 2--knots $\Ss^k(K)$ will likely represent different isotopy classes as $k$ varies.

When $M\cong S^3$, the resulting twist-spun knots $\Ss^n(K)$ have been well studied, starting with Zeeman~\cite{Zeeman_Twisting_1965}, who introduced the general notion (following Artin~\cite{Artin_Zur-Isotopie_1925}).  On the other hand, it appears that very little attention has been focused on the case of twist-spinning knots in non-trivial 3--manifolds.

%%%%%%%%%%%%%%%%%%%%%%%%%%%%%%%%%%%%%%%%%%%%%%%%%%%%%%%%%%%%%%%%%%%%%%%%%%%%%%%%%%
\subsection{Heegaard splittings and trisections}\ 
%%%%%%%%%%%%%%%%%%%%%%%%%%%%%%%%%%%%%%%%%%%%%%%%%%%%%%%%%%%%%%%%%%%%%%%%%%%%%%%%%%

We briefly recall the basic set-up of the theories of Heegaard splittings and trisections.  A \emph{genus $g$ Heegaard splitting} of a closed, connected, orientable 3--manifold $M$ is a decomposition
$$M = H_\delta\cup_\Sigma H_\varepsilon,$$
where $H_\delta$ and $H_\varepsilon$ are handlebodies whose common boundary is a closed surface $\Sigma$ or genus $g$. Every closed 3--manifold admits a Heegaard splitting~\cite{Bing_An-alternative_1959,Moise_Affine_1952}, and any two Heegaard splittings of a fixed manifold are stably equivalent~\cite{Reidemeister_Zur-dreidimensionalen_1933,Singer_Three-dimensional_1933}. 

Let $\delta$ be a collection of $g$ disjoint curves on $\Sigma$ arising as the boundary of $g$ properly embedded disks in $H_\delta$ and satisfying the property that $\Sigma\setminus\nu(\delta)$ is connected and planar.  Let $\varepsilon$ be a similar collection of curves corresponding to $H_\varepsilon$.  The triple $(\Sigma,\delta,\varepsilon)$ is called a \emph{Heegaard diagram} for the splitting $M = H_\delta\cup_\Sigma H_\varepsilon$.  Any two diagrams for a given splitting can be related by handleslides (among the respective sets of curves) and diffeomorphism~\cite{Johannson_Topology_1995}.

A \emph{$(g,k)$--trisection} of a smooth, orientable, connected, closed 4--manifold $X$ is a decomposition $X = X_1\cup X_2\cup X_3$, where 
\begin{enumerate}
	\item each $X_i$ is a four-dimensional 1--handlebody, $\natural^k(S^1\times B^3)$;
	\item for $i\not=j$, each of $X_i\cap X_j$ is a three-dimensional handlebody, $\natural^g(S^1\times D^2)$; and
	\item the common intersection $\Sigma = X_1\cap X_2\cap X_3$ is a closed surface of genus $g$.
\end{enumerate}
The surface $\Sigma$ is called the \emph{trisection surface}, and the \emph{genus} of the trisection is said to be $g = g(\Sigma)$.  The \emph{trisection genus} of a 4--manifold $X$ is the smallest value of $g$ for which $X$ admits a trisection of genus $g$, but no trisection of smaller genus.

Note that $\Sigma$ is a Heegaard surface for $\partial X_i\cong\#^k(S^1\times S^2)$, so $0\leq k\leq g$. As in the case of Heegaard splittings, every smooth 4--manifold admits a trisection, and any two trisections for a fixed 4--manifold are stably equivalent~\cite{Gay-Kirby_Trisecting_2016}.

A \emph{trisection diagram} is a quadruple $(\Sigma, \alpha, \beta, \gamma)$ where each triple $(\Sigma,\alpha, \beta)$, etc., is a Heegaard diagram for $\#^k(S^1\times S^2)$.  As before, any two diagrams corresponding to a given splitting can be made diffeomorphic after  handleslides within each collection of curves. See~\cite{Gay-Kirby_Trisecting_2016,Meier-Schirmer-Zupan_Classification_2016} for complete details.

%The following facts are easy to verify.
%\begin{enumerate}
%	\item The only manifold with trisection genus zero is $S^4$.
%	\item The only manifolds admitting $(1,0)$--trisections are $\CP^2$ and $\overline{\CP^2}$.
%	\item The only manifold admitting a $(1,1)$--trisection is $S^1\times S^3$.
%\end{enumerate}

%%%%%%%%%%%%%%%%%%%%%%%%%%%%%%%%%%%%%%%%%%%%%%%%%%%%%%%%%%%%%%%%%%%%%%%%%%%%%%%%%%
\subsection{Doubly-pointed diagrams}\ 
%%%%%%%%%%%%%%%%%%%%%%%%%%%%%%%%%%%%%%%%%%%%%%%%%%%%%%%%%%%%%%%%%%%%%%%%%%%%%%%%%%

A \emph{doubly-pointed} Heegaard diagram is a tuple $(\Sigma,\delta,\varepsilon,z,w)$, consisting of a Heegaard diagram, together with a pair of base points, $z$ and $w$, in  $\Sigma\setminus\nu(\delta\cup\varepsilon)$.  Suppose the underlying Heegaard diagram describes the 3--manifold $M$.  Then, the base points encode a knot $K$ in $M$ in the following way.  Let $\upsilon_\delta$ and $\upsilon_\varepsilon$ be arcs connecting $z$ and $w$ in $\Sigma\setminus\nu(\delta)$ and $\Sigma\setminus\nu(\varepsilon)$, respectively.  Equivalently, $\upsilon_\delta$ and $\upsilon_\varepsilon$ are boundary parallel arcs contained in the 0--cells of the respective handlebodies.  The knot $K$ is the the union of these two (pushed-in) arcs along their common end points, $z$ and $w$.  The following theorem is standard.

\begin{theorem}
	Given any 3--manifold/knot pair $(M,K)$, there is a doubly-pointed Heegaard diagram describing $(M,K)$.  %Moreover, any two such diagrams can be made diffeomorphic by stabilizing and performing handleslides that miss the base points.
\end{theorem}

A \emph{doubly pointed} trisection diagram is a tuple $(\Sigma,\alpha,\beta,\gamma,z,w)$ where each sub-tuple $(\Sigma,\alpha,\beta,z,w)$, etc., is a doubly pointed Heegaard diagram for $(\#^k(S^1\times S^2), U)$, where $U$ is the unknot.  Suppose the underlying trisection diagram describes the 4--manifold $X$.  Then the base points encode a knotted sphere $\Kk$ in $X$ in the following way.  Let $D_i\subset \partial X_i$ be spanning disks for the three unknots described by the diagram.  Let $\Kk$ be the union of these three disks, after the interiors of the disk have been isotoped to lie in the interiors of the $X_i$.

The decomposition $(X,\Kk) = (X_1,D_1)\cup (X_2,D_2)\cup (X_3,D_3)$ is called a \emph{1--bridge trisection} of the pair $(X,\Kk)$, and $\Kk$ is said to be in \emph{1--bridge position} with respect to the underlying trisection of $X$.  The following results are proved in a forth-coming article with Alex Zupan~\cite{Meier-Zupan_Trisecting_}.

\begin{theorem}
	Let $X$ be a smooth, orientable, connected, closed 4--manifold, and let $\Kk$ be a knotted sphere in $X$.  There exists a trisection of $X$ with respect to which $\Kk$ can be isotoped to lie in 1--bridge position.
\end{theorem}

\begin{corollary}
	For any 4--manifold/2--knot pair $(X,\Kk)$, there is a doubly pointed trisection diagram describing $(X,\Kk)$.
\end{corollary}

%%%%%%%%%%%%%%%%%%%%%%%%%%%%%%%%%%%%%%%%%%%%%%%%%%%%%%%%%%%%%%%%%%%%%%%%%%%%%%%%%%
%%%%%%%%%%%%%%%%%%%%%%%%%%%%%%%%%%%%%%%%%%%%%%%%%%%%%%%%%%%%%%%%%%%%%%%%%%%%%%%%%%
\section{Proof of main theorems}\label{sec:proofs}
%%%%%%%%%%%%%%%%%%%%%%%%%%%%%%%%%%%%%%%%%%%%%%%%%%%%%%%%%%%%%%%%%%%%%%%%%%%%%%%%%%
%%%%%%%%%%%%%%%%%%%%%%%%%%%%%%%%%%%%%%%%%%%%%%%%%%%%%%%%%%%%%%%%%%%%%%%%%%%%%%%%%%

In this section, we give the proofs for the main theorems described in the introduction.  First, we will adopt the Morse 2--function perspective to prove that both the spin and twisted-spin of a 3--manifold admitting a genus $g$ Heegaard splitting admit $(3g,g)$--trisections.

Roughly, for a smooth, orientable, connected, closed 4--manifold $X$, a map $F\colon X\to \R^2$ is a \emph{Morse 2--function} if
\begin{enumerate}
	\item Every regular value $y\in\R^2$ has a neighborhood $D^2$ such that $F$ is projection $S\times D^2\to D^2$ for some closed surface $S$.
	\item The set critical points of $F$ is a smooth one-dimensional submanifold whose image in $\R^2$ is a collection of immersed curves with isolated crossings and semi-cubical cusps.
	\item Every critical value $y\in\R^2$ has local coordinates such that $F$ looks like a generic homotopy of a Morse function:  If $y$ is a cusp, $F$ looks like the birth of a canceling pair of Morse critical points.  If $y$ is a crossing point, $F$ looks like two Morse critical points swapping height. If $y$ is not on a cusp or a crossing point, $F$ looks like a Morse critical point times $I$.
\end{enumerate}
See \cite{Gay-Kirby_Trisecting_2016} for a complete definition.  See also~\cite{Baykur-Saeki_Simplifying_2017} for a detailed overview of various types of generic functions from 4--manifolds to surfaces.

We now sketch a quick, Morse 2--function proof of our first result, which was first conceived by Alex Zupan.  Our proof of Theorem~\ref{thm:diags}, below, will provide a second, independent proof of this result.
 
\begin{reptheorem}{thm:SpunTri}
	Suppose that $M$ admits a genus $g$ Heegaard splitting.  Then each of $\mathcal S(M)$ and $\mathcal S^*(M)$ admits a $(3g,g)$--trisection.
\end{reptheorem}

\begin{proof}
	Let $M$ be a closed, connected, orientable 3--manifold, and suppose that $M$ admits a genus $g$ Heegaard splitting $\Hh$.  Let $f\colon M\to\R$ be a Morse function corresponding to $\Hh$, and suppose that $f$ has isolated critical points of non-decreasing index.

%\begin{figure}[h!]
%	\centering
%	\includegraphics[width=.9\textwidth]{Morse.pdf}
%	\caption{(a) The image of a Morse function $f$ with isolated critical points for a genus $g$ 3--manifold $M$.  (b) The Morse 2--function $\bar F$ on $\bar X=M\times S^1$ induced by $f$. (c) The corresponding Morse 2--function $F$ on the manifold $X$ obtained as surgery on $\bar X$.}
%	\label{fig:Morse}
%\end{figure}

	Consider the 4--manifold $\bar X = M\times S^1$, and let $\bar F\colon X\to\R^2$ be the Morse 2--function induced fiber-wise by the Morse function $f$.  See Figure~\ref{fig:Morse2}(a).  The map $\bar F$ has a single (definite) fold of both indices zero and three, as well as $g$ indefinite folds of both indices one and two.  Note that $\bar F(\bar X)$ is an annulus.  We decorate indefinite folds with arrows that point from the higher genus side of the fold to the lower genus side.

\begin{figure}[h!]
	\centering
	\includegraphics[width=.9\textwidth]{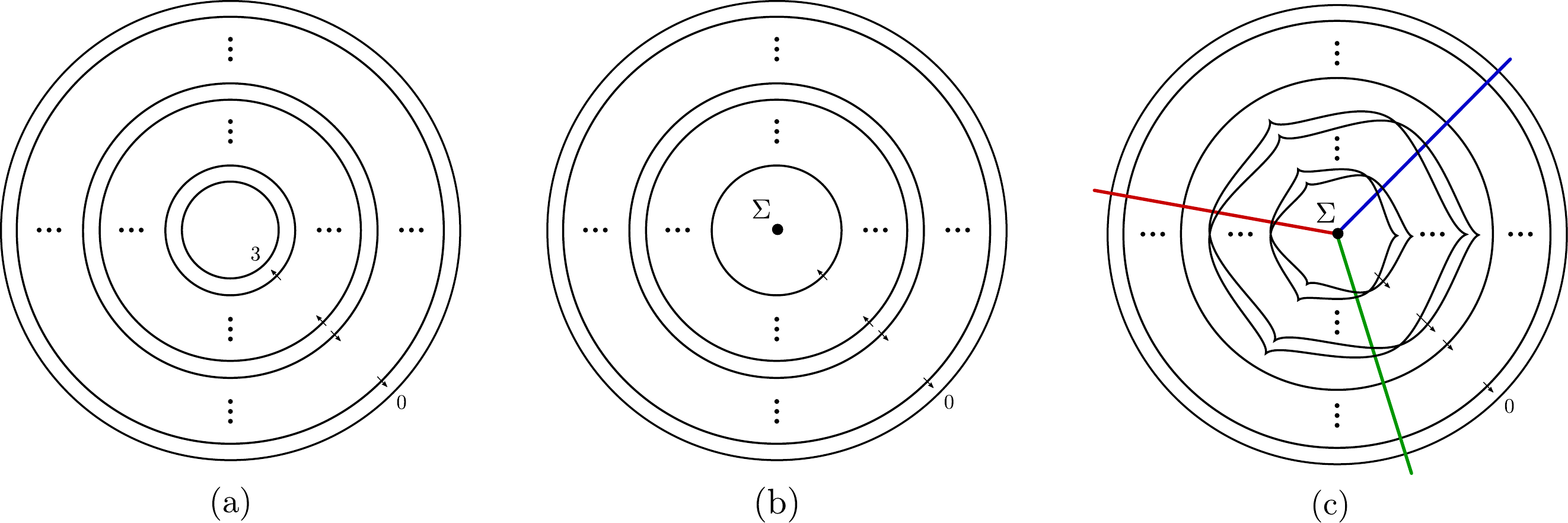}
	\caption{(a) The Morse 2--function $\bar F$ on $\bar X=M\times S^1$ induced by a Morse function $f$ on $M$ with isolated critical points of non-decreasing index. (b) The corresponding Morse 2--function $F$ on the manifold $X$ obtained as surgery on the round three-handle inside $\bar X$. (c) The trisected Morse 2--function homotopic to $F$ with no folds of index two.}
	\label{fig:Morse2}
\end{figure}
	
	Finally, let $X$ denote a 4--manifold obtained from $\bar X$ by surgering out the round three-handle, whose core projects to the fold of index three.  In other words, cut out the $B^3\times S^1$ corresponding the the $h_3\times S^1$, where $h_3$ is the three-handle of $M$, and glue in a copy of $S^2\times D^2$.  In fact, there are two ways to do this~\cite{Gluck_The-embedding_1962}.  One choice results in $\Ss(M)$, the other in $\Ss^*(M)$.  However, this distinction is not visible in the base diagrams of the Morse 2--functions, so we will simply let $X$ denote either choice.
	
	Let $F\colon X\to\R^2$ denote the resulting Morse 2--function, which differs from $\bar F$ in that it has no (definite) fold of index three, and $F(X)$ is a disk.  See Figure~\ref{fig:Morse2}(b).  Note that the fiber $\Sigma$ over the central point of the disk is a two-sphere. To complete the proof, we will homotope $F$, using standard moves, until it has no folds of index two or greater.  To do this, we will take each fold of index two and transform it into an immersed fold of index one containing six cusps.  We can do this one index two fold at a time, and we illustrate this sub-process in Figure~\ref{fig:Homotope}.

\begin{figure}[h!]
	\centering
	\includegraphics[width=.9\textwidth]{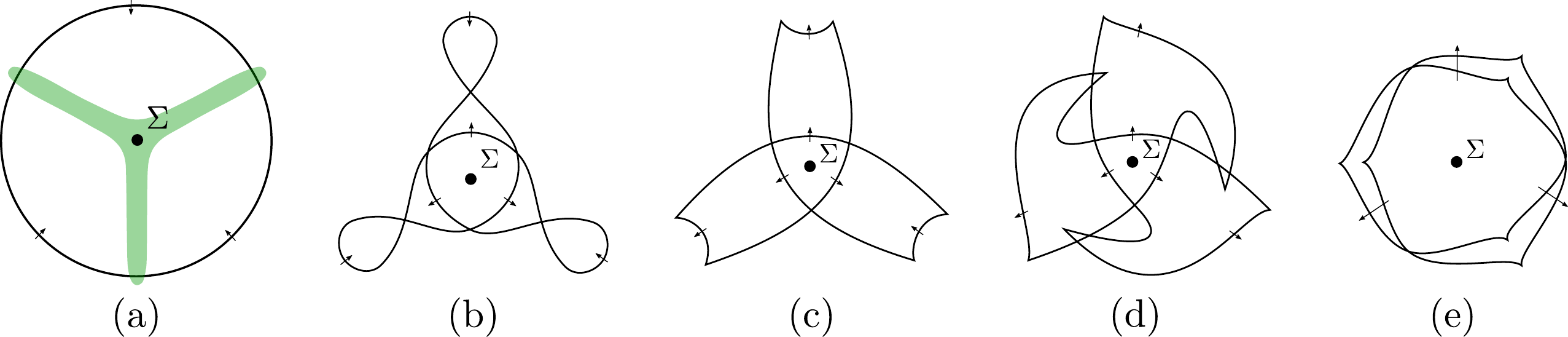}
	\caption{The process (from left to right) of turning a index two fold inside out.  Arrows indicate the direction of decrease of the fiber genus.}
	\label{fig:Homotope}
\end{figure}

	First, we select three points on the index two fold and drag them radially towards and past the center point, this can be seen as a sort of a contraction of the shaded area in Figure~\ref{fig:Homotope}(a), which results in Figure~\ref{fig:Homotope}(b).  This is accomplished via a $\text{R2}_0$ move followed by a $\text{R3}_3$ move.  (See~\cite{Baykur-Saeki_Simplifying_2017} for details.  All base diagram moves employed here are always-realizable.) Next, we turn each of the three kinks into a pair of cusps, resulting in Figure~\ref{fig:Homotope}(c). This can be accomplished via three instances of the flip move, each followed by a $\text{R2}_2$ move.  Note that the genus of $\Sigma$ has been increased by three.  Figure~\ref{fig:Homotope}(d) follows via three $C$-moves, and Figure~\ref{fig:Homotope}(e) follows after three $\text{R2}_2$ moves.
	
	After the above process has been carried out on the innermost indefinite fold of index two in Figure~\ref{fig:Morse2}(b), the resulting six-cusped fold can be pushed outward, past the indefinite folds of index two.  To pass each such fold, we require six instances of the $C$-move, followed by three $\text{R3}_3$ moves, followed by six $\text{R2}_2$ moves. Then, the above process can be repeated for each indefinite fold of index two, resulting in the simplified diagram shown in Figure~\ref{fig:Morse2}(c).

	Note that the fiber $\Sigma$ of the central point now has genus~$3g$.  Choose three rays as in Figure~\ref{fig:Morse2}(c): The preimages of these rays are genus $3g$ handlebodies, which intersect at their common boundary, $\Sigma$.  Similarly, the preimages of the regions between the rays are diffeomorphic to $\natural^g(S^1\times B^3)$. (Each such region is the thickening of a three-dimensional handlebody union  $2g$ three-dimensional two-handles that are attached along primitive curves.)  Therefore, we have a $(3g,g)$--trisection of $X$, as desired.

\end{proof}

Note that the base diagram in Figure~\ref{fig:Morse2}(c) is a simplification of the original base diagram, but is \emph{not} ``simple'' in the sense of~\cite{Baykur-Saeki_Simplifying_2017}. This raises the following question.

\begin{question}
	Does every four-manifold admit a Morse 2--function whose base diagram consists of a disjoint union of indefinite folds of index one, some of which are embedded with no cusps and the rest of which are are immersed with six cusps and three double points, as in Figure~\ref{fig:Homotope}(e)?
\end{question}

%%%%%%%%%%%%%%%%%%%%%%%%%%%%%%%%%%%%%%%%%%%%%%%%%%%%%%%%%%%%%%%%%%%%%%%%%%%%%%%%%%
\subsection{From Heegaard diagrams to trisection diagrams}\label{subsec:diags}\ 
%%%%%%%%%%%%%%%%%%%%%%%%%%%%%%%%%%%%%%%%%%%%%%%%%%%%%%%%%%%%%%%%%%%%%%%%%%%%%%%%%%

Next, we show how, given a Heegaard diagram for a 3--manifold $M$, one can produce a trisection diagram for either $\Ss(M)$ of $\Ss^*(M)$.  Though the distinction between this pair of 4--manifolds was not visible from the Morse 2--function perspective, these manifolds are not, in general, diffeomorphic, so they will necessarily be described by different trisection diagrams.

\begin{reptheorem}{thm:diags}
	Let $(S,\delta,\varepsilon)$ be a genus $g$ Heegaard diagram for a closed 3--manifold $M$ with the property that $H_\varepsilon$ is standardly embedded in $S^3$.  Then,
	\begin{enumerate}
		\item the 4--manifold $\Ss(M)$ admits a trisection diagram that is obtained from $(S,\delta,\varepsilon)$ via the local modification at each curve of $\varepsilon$ shown in Figure~\ref{fig:LocalDiag}, and
		\item the 4--manifold $\Ss^*(M)$ admits a trisection diagram that is obtained from $(S,\delta,\varepsilon)$ via the local modification at each curve of $\varepsilon$ shown in Figure~\ref{fig:LocalDiag*}.
	\end{enumerate}
\end{reptheorem}

Note that the condition on $H_\varepsilon$ is equivalent to the condition that $(S,\delta,\varepsilon)$ be drawn as in Figure~\ref{fig:StdHeeg}.

\begin{figure}[h!]
	\centering
	\includegraphics[width=.35\textwidth]{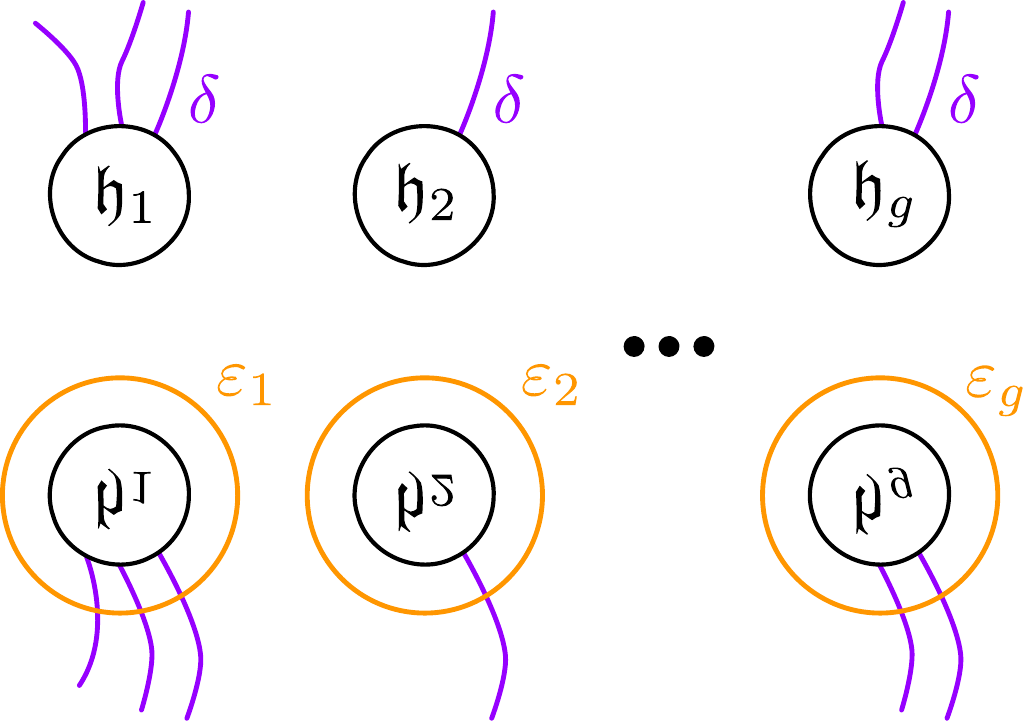}
	\caption{A suitable Heegaard diagram; the $\varepsilon$--curves bound obvious disks in the plane.}
	\label{fig:StdHeeg}
\end{figure}

\begin{proof}

We'll first discuss the the spin $\Ss(M)$, then modify the argument to address the twisted-spin $\Ss^*(M)$.

Let $M=H_\delta\cup_S H_\varepsilon$ be a genus $g$ Heegaard splitting for $M$.  We have the following decomposition:
$$\Ss(M) = (H_\delta\times S^1)\cup_Y(\Ss(H_\varepsilon)),$$
where $Y = S\times S^1$.  This decomposition is visible in Figure~\ref{fig:Morse2}(b), where $Y$ is the preimage of a circle separating the indefinite folds of index one from those of index two.  In the proof of Theorem \ref{thm:SpunTri} above, the Morse 2--function was modified on $\Ss(H_\varepsilon)$ in such a way that the central fiber became a genus $3g$ surface $\Sigma$.  Our first task is to identify $\Sigma$ inside $\Ss(H_\varepsilon)$.  Our approach will be to work from Figure~\ref{fig:Morse2}(b), beginning at the center, and ``trisect'' each subsequent index two fold.

The space $\Ss(H_\varepsilon)$ can be obtained from $S^2\times D^2$ by attaching $g$ round one-handles in the following manner.  We will parameterize $D^2$ by $(r,\theta)$ with $r\in[0,1]$ and $\theta\in S^1\subset \C$, and we will let $\vec r_\theta\subset D^2$ denote the unit-length segment at angle $\theta$.  For $i=1,2,\ldots, g$, let $D^+_i$ and $D^-_i$ be a pair of disjoint disks on $S^2$, and attach a three-dimensional one-handle $\frak h^\theta_i$ to $S^2\times\vec r_\theta$ along $D^\pm_i\times\{(1,\theta)\}$ for each $\theta\in S^1$.  (For each $i$, the union $\frak h_i = \bigcup_\theta\frak h^\theta_i$ is a four-dimensional round one-handle.)  Equivalently, we can view this handle attachment as the identification of $D^+\times\{(1,\theta)\}$ with $D^-\times\{(1,\theta)\}$ via a reflection (conjugation) map. We parameterize $D^\pm_i$ by $(s,\phi)$, where $s\in[0,1]$ and $\phi\in S^1\subset \C$, and we let 
$$\omega^\theta_i(s,\phi) = ((s,\phi)\times\vec r_\theta)\cup(\overline{(s,\phi)}\times\vec r_\theta).$$
In other words, the $\omega^\theta_i(s,\phi)$ are arcs that run over $\frak h^\theta_i$, connecting identified pairs of points in $D^\pm$ on $S^2\times\{(0,0)\}$.

Consider the arcs $\omega_i^\theta$ given by
$$\omega^\theta_i = \omega^\theta_i(1/2,\theta) = ((1/2,\theta)\times\vec r_\theta)\cup(\overline{(1/2,\theta)}\times\vec r_\theta).$$
In other words, $\omega^\theta_i$ is an arc running over $\frak h_i^\theta$ connecting the point with angle $\theta$ on the circle of radius $1/2$ on $D_i^+$ to the conjugate point on $D_i^-$.  Note that $\frak h_i^\theta$ can be regarded as a regular neighborhood of $\omega_i^\theta$, so $\Ss(H_\varepsilon)$ is a regular neighborhood of the two-complex
$$S^2\cup\left(\bigcup_{i=1}^g\bigcup_{\theta\in S^1}\omega_i^\theta\right).$$

Consider the three angle values $\theta_j = \frac{2\pi}{3}j$, for $j=0,1,2$, along with the $3g$ arcs $\omega_i^{\theta_j}$. Let $\Sigma$ be the surface obtained by surgering the central $S^2$ along these $3g$ arcs.  Note that $\Sigma$ has genus $3g$ and is contained in the interior of $\Ss(H_\varepsilon)$.  We now describe three compression bodies whose higher genus boundary component coincides with $\Sigma$ and whose lower genus boundary component is a fiber of $Y = S\times S^1$, hence has genus $g$.  Thus, we must describe $2g$ compression disks for each compression body.

Let $\frak h_i^j$ denote a small tubular neighborhood of $\omega_i^{\theta_j}$.  We can think of $\frak h_i^j$ as a small three-dimensional one-handle inside the larger three-dimensional one-handle $\frak h_i^{\theta_j}$, as in Figure~\ref{fig:Handles}.  Let $\Delta_{1,i}^j$ denote the cocore of $\frak h_i^j$.  Next, notice that $\Sigma\cap D_i^+$ is a thrice-punctured disk.  These punctures cut the circle of radius $1/2$ in $D_i^+$ into three arcs.  Call these arcs $a_i^j$, with the value $j$ determined by the property that $a_i^j\cap\frak h_i^j = \emptyset$.  See Figure~\ref{fig:Handles}(a).  Let $\Delta_{2,i}^j$ be the union of the arcs $\omega_i^\theta$ corresponding to the points in arc $a_i^j$.  Note that the $\Delta_{2,i}^j$ are compression disks for $\Sigma$.

\begin{figure}[h!]
	\centering
	\includegraphics[width=.8\textwidth]{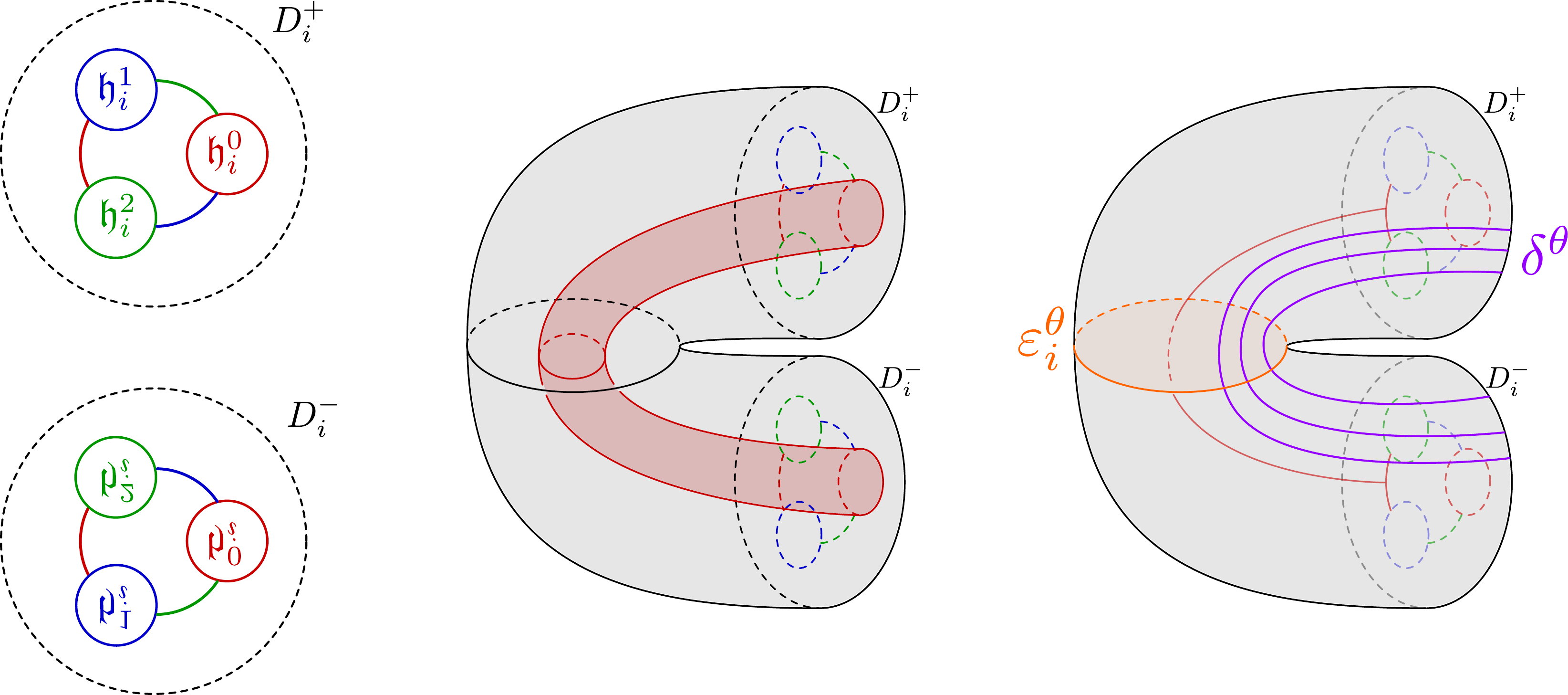}
	\caption{(a) The disk $D^\pm_i$ on the central sphere $S^2\times\{(0,0)\}$ describing the attaching region for $\frak h_i^\theta$. (b) The handle $\frak h^0_i$ inside the $\frak h^{\theta_0}_i$, and the portion of $H_\alpha$ bounded thereby. (c) The handle $\frak h^\theta_i$ for some $\theta\in(2\pi/3,4\pi/3)$.  In the interior, we have the arc $\omega^\theta_i$, which lies in the $\alpha$--disk $\Delta^0_{2,i}$. One the boundary, we have the curve $\varepsilon^\theta_i$ and potions of the curves from $\delta$, which serve to parameterize the genus $g$ surface $S\times\theta$ in $\partial (\Ss(H_\varepsilon)) = \partial (H_\delta\times S^1) = S\times S^1$.}
	\label{fig:Handles}
\end{figure}

Let $H^j$ denote the compression body defined by the disks $\{\Delta_{1,i}^j,\Delta_{2,i}^j\}_{i=1}^g$. Note that $\Sigma$ is contained in the union
$$S^2\cup\left(\bigcup_{i=1}^g\bigcup_{j=0}^2\frak h_i^{\theta_j}\right).$$  If we compression $\Sigma$ using, say, the disks $\Delta_{1,i}^0$, then the resulting surface can be made disjoint from the handles at angle 0.  Slightly differently, if we compress further using the disks $\Delta_{2,i}^0$, then $\Sigma$ can be isotoped to lie in any single angle, say $2\pi/3$.  It follows that the result of compressing $\Sigma$ along the disks of $\Delta_{1,i}^0$ and $\Delta_{2,i}^j$ is the surface $S\times\{2\pi/3\}$.  Repeating this, we see that the lower genus boundary component of $H^j$ can be assumed to be $S\times\{\theta_j+2\pi/3\}$, as desired.

Consider the complex $X=\Sigma\cup H^0\cup H^1\cup H^2$.  This complex is a three-dimensional neighborhood of the two-complex described above.  It follows that $\Ss(H_\varepsilon)$ is obtained by thickening $X$.

We complete the $H^j$ to handlebodies by attaching a copy of $H_\delta$ to the lower genus boundary component.  For example, we let $H_\alpha = H^0\cup(H_\delta\times\{2\pi/3\})$, and we obtain $H_\beta$ and $H_\gamma$ from $H^1$ and $H^2$ similarly.  We claim that $H_\alpha\cup H_\beta\cup H_\gamma$ is the spine of a trisection of $\Ss(M)$. A regular neighborhood of this spine is given by $\Ss(H_\varepsilon)$ plus thickening of the three $H_\delta$--fibers.  All that remains is to fill in the four dimensional spans between the $H_\delta$--fibers.  Each of these pieces is $H_\delta\times I$, which is a four-dimensional one-handlebody.  If follows that this spine defines a $(3g,g)$--trisection of $\Ss(M)$.

Finally, we will describe a trisection diagram corresponding to this spine by describing the curves $\alpha$ lying on $\Sigma$ that determine the handlebody $H_\alpha$.  The construction is symmetric in $\alpha$, $\beta$, and $\gamma$, so the description of the other curves will follow. Recall that we assumed that the Heegaard diagram $(S,\delta,\varepsilon)$ was standard, as in Figure~\ref{fig:StdHeeg}.  Figure~\ref{fig:Handles} shows how to take each hand $\frak h_i$ and create from it a triple of handles, $\frak h_i^j$, as in the construction of the trisection above. For each $i$, two $\alpha$ disks are obtained.  Let $\alpha_{g+i} = \partial \Delta_{1,i}^0$, and let $\alpha_{2g+i} = \partial\Delta_{2,i}^0$.  See Figure~\ref{fig:Local_alpha}. Compressing along these $2g$ disks gives the fiber $S\times\{2\pi/3\}$.

\begin{figure}[h!]
	\centering
	\includegraphics[width=.4\textwidth]{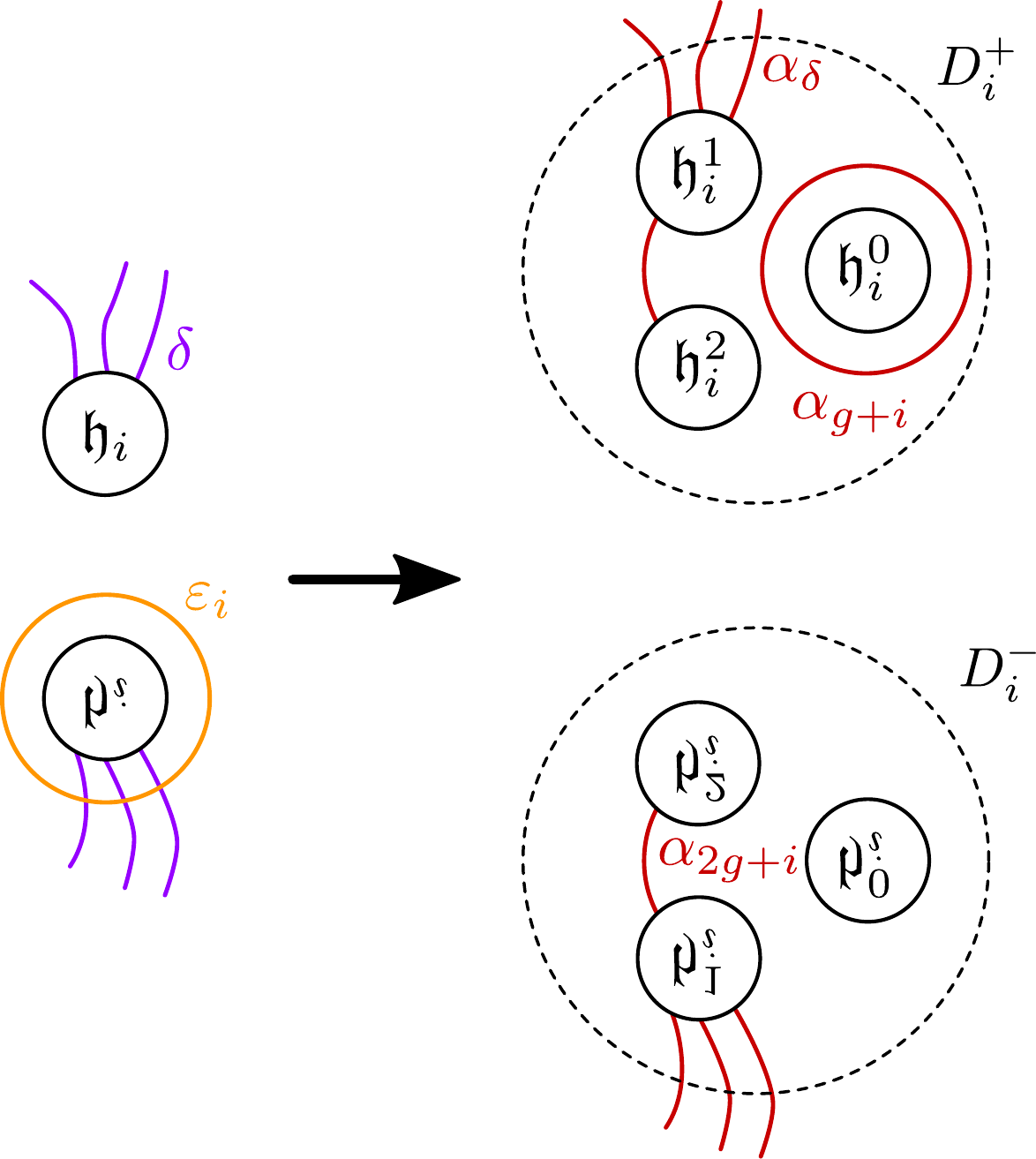}
	\caption{The local transition from a Heegaard diagram $(\delta,\varepsilon)$ to the $\alpha$--curves of the trisection diagram $(\alpha,\beta,\gamma)$. The $\beta$-- and $\gamma$--curves are obtained in a symmetric way.}
	\label{fig:Local_alpha}
\end{figure}

Figure~\ref{fig:Handles}(c) shows one $\theta$--slice of the round handle $\frak h_i$.  At each such $\theta$--slice, we see $\varepsilon_i^\theta$ bounding to the inside, while the curves of $\delta^\theta$ run over the handle as prescribed by the original diagram (Figure~\ref{fig:StdHeeg}).  Imagine $\theta=2\pi/3$ here, and recall that we think of $\frak h^{2\pi/3}_i$ as a neighborhood of $\omega_i^{2\pi/3}$ (the arc shown in Figure~\ref{fig:Handles}(c)).  The disks bounded by the curve $\delta$ in $H_\delta\times\{2\pi/3\}$ are almost the remaining $\alpha$--disks, but their boundary lies on the lower genus boundary component of the compression body $H^0$, not on $\Sigma$.  However, it is a simple matter to flow the boundaries of this disk up through the compression body (using the vertical structure) until they lie on $\Sigma$.

Thus, for $i=1,\ldots, g$, $\alpha_i$ will be determined by $\delta_i$ in the following way.  Outside of the $D_i^\pm$, $\alpha_i$ coincides with $\delta_i$.  Inside, the arcs run from $\partial D^\pm_i$ to the handle $\frak h_i^1$.  In fact, this choice is well-defined, thanks to the presence of the curves $\alpha_{g+i}$ and $\alpha_{2g+i}$, as in Figure~\ref{fig:Local_alpha}.  Let $\alpha_\delta = \{\alpha_1,\ldots,\alpha_g\}$.  

\begin{figure}[h!]
	\centering
	\includegraphics[width=\textwidth]{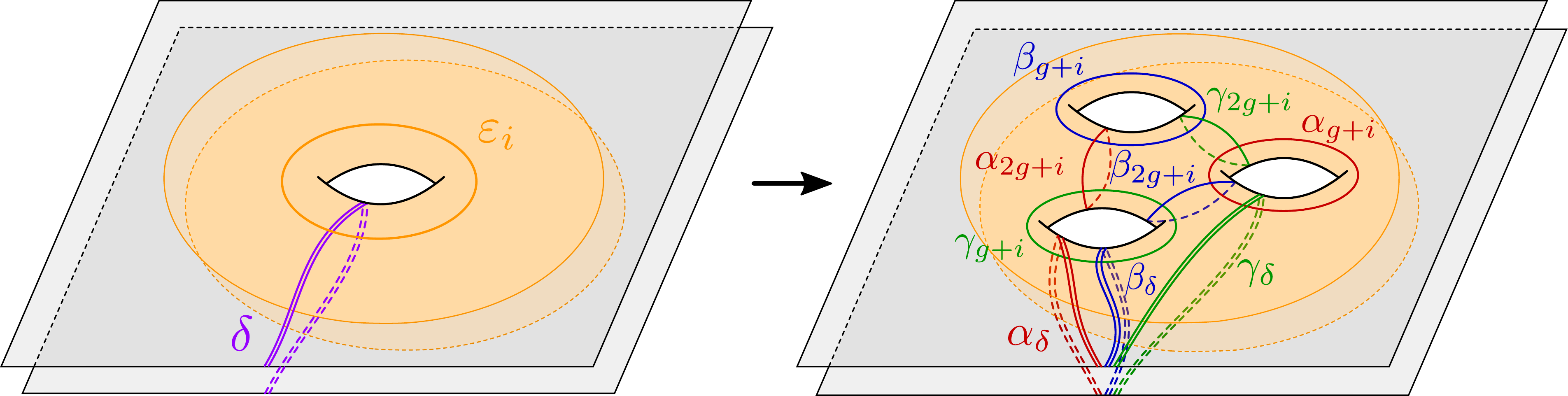}
	\caption{The local modification used to transform a Heegaard diagram $(\delta,\varepsilon)$ for a 3--manifold $M$ into a trisection diagram $(\alpha, \beta, \gamma)$ for the spun manifold $\Ss^*(M)$.}
	\label{fig:LocalDiag}
\end{figure}

The sum total of this local modification is shown in Figure~\ref{fig:LocalDiag}.  Note that the curves $\alpha_\delta$, $\beta_\delta$, and $\gamma_\delta$ coincide after compressions of the other types of curves.  This reflects the fact that these curves come from $H_\delta\times S^1$. This completes the proof of part (1).

To pass from the case of $\Ss(M)$ to that of $\Ss^*(M)$, we will perform a Gluck twist on the central $S^2$, cutting out a $S^2\times D^2$ neighborhood and re-gluing with a full twist.  Importantly, we assume that the twisting takes place in the $\theta$--interval $[0,2\pi/3]$.  Under this assumption, we see that $\Sigma$ is preserved after the Gluck twist, as are $H_\alpha$ and $H_\beta$. Further, the $\gamma_\delta$ and $\gamma_{g+i}$ are also preserved.  The only change occurs to the curves $\gamma_{2g+i}$; the Gluck twist is concentrated above the arc $a_i^2$.  The disks $\gamma_{2g+1}$ sitting above these arcs get twisted around the terminal locus of the arc. In terms of the diagram, this gluing amounts to performing a Dehn twist of the $\gamma_{2g+i}$ about the corresponding $\beta_{g+i}$.  Thus, Figure~\ref{fig:LocalDiag} changes to Figure~\ref{fig:LocalDiag*}.  This completes the proof of part (2).

\begin{figure}[h!]
	\centering
	\includegraphics[width=\textwidth]{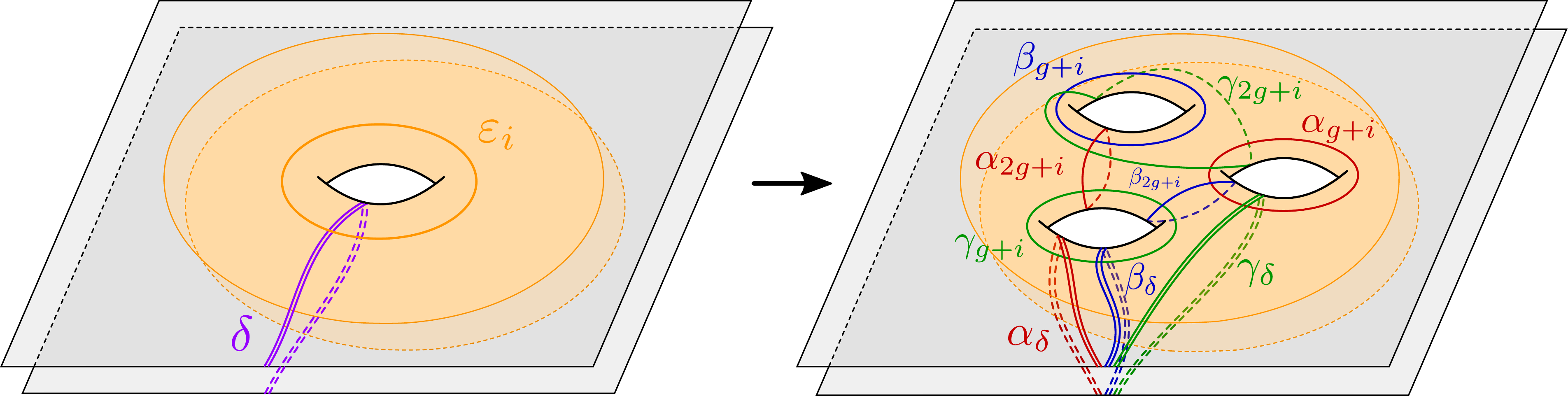}
	\caption{The local modification used to transform a Heegaard diagram $(\delta,\varepsilon)$ for a 3--manifold $M$ into a trisection diagram $(\alpha, \beta, \gamma)$ for the twist-spun manifold $\Ss^*(M)$.}
	\label{fig:LocalDiag*}
\end{figure}

\end{proof}

Note that within the above proof, we have also given a second proof of Theorem~\ref{thm:SpunTri} that is independent of the original Morse 2--function proof.

%%%%%%%%%%%%%%%%%%%%%%%%%%%%%%%%%%%%%%%%%%%%%%%%%
\subsection{Doubly-pointed diagrams}\ 
%%%%%%%%%%%%%%%%%%%%%%%%%%%%%%%%%%%%%%%%%%%%%%%%%

Let $M$ be a closed, connected, orientable 3--manifold, and let $K$ be a knot in $M$.  Let $M=H_1\cup_S H_2$ be a Heegaard splitting for $M$.  Assume that $S$ has large enough genus (stabilizing if necessary) so that $K$ can be put in 1--bridge position with respect to $S$.  This means that $\upsilon_i=K\cap H_i$ is a properly embedded, boundary-parallel arc for $i=1,2$.  Let $\{z,w\} = K\cap S$, and assume that $\upsilon_1$ is contained in the zero-handle $\frak h_0$, while $\upsilon_2$ is contained in the three-handle $\frak h_3$.

\begin{reptheorem}{thm:dpdiags}
	Let $(S,\delta,\varepsilon)$ be a genus $g$ Heegaard diagram for a closed 3--manifold $M$ with the property that $H_\varepsilon$ is standardly embedded in $S^3$.  Let $K$ be a knot in $M$ such that $(S,\delta,\varepsilon,z,w)$ is a doubly-pointed Heegaard diagram for the pair $(M,K)$. Then the pairs $\Ss^n(M,K)$  admit doubly-pointed trisection diagrams that are obtained from $(S,\delta,\varepsilon,z,w)$ via a local modification at each curve of~$\varepsilon$.
\end{reptheorem}

\begin{proof}
	By the last part of the proof of Theorem~\ref{thm:diags}, it is clear that gluing using $\tau^n$ corresponds to Dehn twisting $\gamma_{2g+i}$ $n$ times about $\beta_{g+i}$.  Thus, the underlying trisection diagram $(\Sigma,\alpha,\beta,\gamma)$ results from the same local modification as in Figure~\ref{fig:LocalDiag*}, except with the added Dehn twists.
	
	It remains to show that $\Ss^n(K)$ is in 1--bridge position with respect to this trisection, so we verify that $\Ss^n(K)$ intersects the three handlebodies in boundary parallel arcs and intersects the four-dimensional pieces in boundary parallel disks.
	
	The sphere $\Ss^n(K)$ can be decomposed as $$D^2\times\{N\}\cup (\upsilon_1\times S^1)\cup D^2\times\{S\}.$$  We now consider how the various parts of this decomposition intersect the trisection of $\Ss^n(M)$.
	
	Consider $\upsilon_1\times S^1\subset H_\delta\times S^1$.  This annulus intersects each of fibers in an arc.  For example, $\upsilon_1\times\{2\pi/3\}$ is an arc in $H_\delta\times\{2\pi/3\}$ with endpoints in the lower genus boundary component, $S\times \{2\pi/3\}$, of the compression body $H^0$.  The endpoints of this arc are $\{z,w\}\times\{2\pi/3\}$. Since $\upsilon_1$ is boundary parallel (in $H_\delta\subset M$) into $S$, we have that $\upsilon_1\times\{2\pi/3\}$ is boundary parallel (in $H_\delta\times\{2\pi/3\}\subset H_\alpha$) into $S\times\{2\pi/3\}$ and that the disk $\upsilon_1\times[0,2\pi/3]$ is boundary parallel (in $H_\delta\times[0,2\pi/3]$) into $S\times[0,2\pi/3]$.
	
	Let us focus now on $\upsilon_\alpha = \Ss^n(K)\cap H_\alpha$, recalling that $H_\alpha = H_\delta\times\{2\pi/3\}\cup_{S\times\{2\pi/3\}} H^0$. We have already seen that $\Ss^n(K)\cap (H_\delta\times\{2\pi/3\}) = \upsilon_1\times\{2\pi/3\}$ is boundary parallel into $S\times\{2\pi/3\}$.  Next, we note that $\Ss^n(K)\cap H^0$ is simply two arcs.  One arc runs from $\{z\}\times\{2\pi/3\}$ to the north pole $N$ of the sphere $S^2\times\{0\}$ that was the core of the original filling in the twist-spinning operation.  Of course, this sphere was stabilized to produce the trisection surface $\Sigma$, but these modification were performed away from the poles.  Thus, this arc is vertical in the compression body $H^0$.  Similarly, the second arc is vertical and connects $\{w\}\times\{2\pi/3\}$ to the south pole $S$ of $\Sigma$.  Since $\Sigma$ and $S\times\{2\pi/3\}$ cobound the compression body $H^0$ and $\upsilon_\alpha$ is a flat arc in the lower genus side together with two vertical arcs, it follows that $\upsilon_\alpha$ can be isotoped to lie in $\Sigma$, as desired.  The same goes for the arcs $\upsilon_\beta$ and $\upsilon_\gamma$.
	
	Next, let us focus on the 4--dimensional region $X_3$ between $H_\alpha$ and $H_\gamma$.  Recall that $H_\gamma = H^2\cup H_\delta\times\{0\}$, so we can write
	$$X_3 = \left(H_\delta\times[0,2\pi/3]\right)\cup_{S\times[0,2\pi/3]}\left((H^0\cup_\Sigma H^2)\times I\right).$$
	The second piece of the union comes from the fact that $\Ss^n(H_\varepsilon)$ was seen to be a thickening of the complex $\Sigma\cup H^0\cup H^1\cup H^2$. Now, we note that $\Dd_3=\Ss^n(K)\cap X_3$ is simply the disk $\upsilon_1\times[0,2\pi/3]$, which we have already observed is boundary parallel into $S\times[0,2\pi/3]$, together with some vertical pieces in the thickening $(H^0\cup_\Sigma H^2)\times I$.
	
	Since $\partial \Dd_3 = \upsilon_\alpha\cup_{\{N,S\}}\upsilon_\gamma$, once we have pushed most of $\Dd_3$ into $S\times[0,2\pi/3]$, we can use the product structure of $(H^0\cup_\Sigma H^2)\times I$ and the boundarly parallelism of $\upsilon_\alpha$ and $\upsilon_\gamma$ to push $\Dd_3$ into $H^0\cup_\Sigma H^2\subset H_\alpha\cup_\Sigma H_\gamma$, as desired.  The same goes for the other 4--dimensional pieces $(X_2,\Dd_2)$ and $(X_1,\Dd_1)$.
	
	Thus, $\Ss^n(K)$ is in 1--bridge position with respect to the trisection described in the proof of Theorem~\ref{thm:diags}.  Note that the local modification require here is slightly different:  We must twist the $\gamma_{2g+i}$ around the $\beta_{g+i}$ a total of $n$ times.  However, once we have done that, we have a doubly-pointed diagram for $\Ss^n(M,K)$; since the double-point $\{z,w\}$ is distant from the $\varepsilon_i$, it is not affected by the modification, and it becomes the doubly-point $\{N,S\}$ for the doubly-pointed trisection diagram.  This completes the proof.  (In order to see that $\{N,S\} = \{z,w\}$ in the appropriate manner, we simply treat the original surface $S$ as the boundary of the result of attaching handles to $S^2\times\{0\}$ in the standard way.  In other words, if we think of the original double-point $\{z,w\}$ as the ``poles'' of $S$, the the new double-point $\{z,w\} = \{N,S\}$ for $\Sigma$ is simply the ``poles'' of $\Sigma$ coming from the poles of $S^1\times\{0\}$.)
	
%	By our remarks just above, since $\upsilon_1$ is boundary parallel (in $H_\delta\subset M$) into $S$, we have that 
%	
%	
%	 In the same way that we extended the disks with boundary $\delta$ to give $\alpha_1,\ldots, \alpha_g$ using the vertical structure of the compression body $H^0$, we can also extend this arc via arcs that connect $\{z,w\}\times\{2\pi/3\}$ to $\{z,w\}\times (0,0)$.
%	
%	This same vertical structure shows that this arc is boundary parallel into $\Sigma$.  Similarly, since $\upsilon_1$ is boundary parallel into $S$, we have that $\upsilon_1\times[0,2\pi/3]$ is boundary parallel into $S\times[0,2\pi/3]$.  Once this is achieved, we are done, since the three manifold $H_\alpha\cup H_\beta$ can also be described as
%	$$(H_\delta\times\{0\})\cup S\times[0,2\pi/3]\cup(H_\delta\times\{2\pi/3\}).$$
%	The Heegaard splitting $H_\alpha\cup_\Sigma H_\beta$ is a genus $3g$ splitting for the manifold $\#^g(S^1\times S^2)$.  This latter description is the genus $g$ splitting for the same manifold -- the stabilizations have been removed with this change of perspective.
%	
%	Applying this argument to all three handlebodies completes the proof.
\end{proof}

%%%%%%%%%%%%%%%%%%%%%%%%%%%%%%%%%%%%%%%%%%%%%%%%%%%%%%%%%%%%%%%%%%%%%%%%%%%%%%%%%%
%%%%%%%%%%%%%%%%%%%%%%%%%%%%%%%%%%%%%%%%%%%%%%%%%%%%%%%%%%%%%%%%%%%%%%%%%%%%%%%%%%
\section{Corollaries, Examples, and Questions}\label{sec:exs}
%%%%%%%%%%%%%%%%%%%%%%%%%%%%%%%%%%%%%%%%%%%%%%%%%%%%%%%%%%%%%%%%%%%%%%%%%%%%%%%%%%
%%%%%%%%%%%%%%%%%%%%%%%%%%%%%%%%%%%%%%%%%%%%%%%%%%%%%%%%%%%%%%%%%%%%%%%%%%%%%%%%%%

Let us return to the question of classifying manifolds with low trisection genus. The following facts are easy to verify.
\begin{enumerate}
	\item The only manifold with trisection genus zero is $S^4$.
	\item The only manifolds with trisection genus one are $\CP^2$, $\overline{\CP^2}$, and $S^1\times S^3$.
\end{enumerate}
Moreover, $S^2\times S^2$ is the only irreducible four-manifold with trisection genus two. We also have the following.

\begin{proposition}\label{prop:facts}
	Suppose $X$ admits a $(g,k)$--trisection.  Then,
	\begin{enumerate}
		\item $\chi(X) = 2+g-3k$.
		\item $\pi_1(X)$ has a presentation with $k$ generators.
		\item $|H_1(X;\Q)|\leq k$ and $|H_2(X;\Q)|\leq g-k$.
	\end{enumerate}
\end{proposition}

\begin{proof}
	Such an $X$ admits a handle decomposition with a single 0--handle, $k$ 1--handles, $g-k$ 2--handles, $k$ 3--handles, and a single 4--handle~\cite{Gay-Kirby_Trisecting_2016,Meier-Schirmer-Zupan_Classification_2016}.
\end{proof}

We can now prove Corollary~\ref{coro:Minimal}.  Note that $(g,k)$--trisections are standard if $k\geq g-1$~\cite{Meier-Schirmer-Zupan_Classification_2016}.

\begin{repcorollary}{coro:Minimal}\ 
	For every integer $g\geq 3$ and every $1\leq k\leq g-2$, there exist infinitely many distinct 4--manifolds admitting minimal $(g,k)$--trisections.  
\end{repcorollary} 

\begin{proof}
	Let $k\geq 1$, and let $M$ be a three-manifold with Heegaard genus $g(M) = k$ and $rk(\pi_1(M))=k$.  Let $X = \Ss(M)$.  By Theorem~\ref{thm:SpunTri}, $X$ admits a $(3k,k)$--trisection.  By Proposition~\ref{prop:facts}(2), since $\pi_1(X) = \pi_1(M)$, $X$ cannot admit a $(g',k')$--trisection with $k'<k$.  By Proposition~\ref{prop:facts}(1), $X$ cannot admit a $(g',k)$--trisection with $g'<g$.

	Now, let $X_n = X\#(\#^n\CP^2)$, which admits a $(3k+n,k)$--trisection.  By similar reasoning, the second parameter, $k$,  cannot be decreased, nor can the first parameter, $g=3k+n$. To complete the proof, we let $M$ be a connected sum of $k$ lens spaces, so $M$ satisfies the necessary hypotheses of $g(M) = rk(\pi_1(M))=k$.
\end{proof}

Conspicuously absent from this result is the case of $k=0$.

\begin{question}
	For some $g\geq 3$, are there infinitely many four-manifolds admitting (minimal) $(g,0)$--trisections?
\end{question}

Since the classification of four-manifolds with trisection genus three remains open, we next turn our attention to the case of spun lens spaces.

%%%%%%%%%%%%%%%%%%%%%%%%%%%%%%%%%%%%%%%%%%%%%%%%%
\subsection{Spinning lens spaces}\ 
%%%%%%%%%%%%%%%%%%%%%%%%%%%%%%%%%%%%%%%%%%%%%%%%%

Figure~\ref{fig:L52} shows how to obtain a trisection diagram for $\Ss_5$.  The process is general.  Start with the genus one Heegaard diagram $(\delta,\varepsilon)$ for $L(p,q)$ where $\varepsilon$ is drawn as the boundary of the disk filling the center hole, and the curve $\delta$ is a $(p,q)$--curve. After performing the local modification, we see the characteristic 6--tuple of curves in the center, encircled by three copies of something similar to a $(p,q)$--curve.  In fact, these three more complicated outer curves will become $(p,q)$--curves (and will coincide) after the compression of any pair of same colored curve in the center.  Let $\Tt(p,q)$ denote the trisection obtained in this way.

%\begin{remarks}\ 
%\begin{enumerate}
%	\item It is possible to slide any one of the outer curves to coincide with any other of the outer curves.  This requires a sequence of $p$ slides in each case.  It follows that the $L$--invariant is bounded above by $3p$ for $\Ss_p$.
%	\item The trisections $\Tt(p,q)$ (and in general, any spun trisection) is \emph{weakly-reducible}; there is a triple of curves, one of each color, whose members are pair-wise disjoint.  This allows the trisection to be \emph{un-telescoped}, resulting in an essential surface, which, in this case, is simply the two-sphere $S^2\times\{\ast\}$ at the center of the original $S^2\times D^2$.  It is essential, since it bounds $M^\circ$.
%\end{enumerate}	
%\end{remarks}

By Corollary~\ref{coro:lens}, we know that $\Ss(L(p,q))$ and $\Ss^*(L(p,q))$ are diffeomorphic to $\Ss_p$, independent of $q$ and $q'$. This raises the following question.

\begin{question}
	Are $\Tt(p,q)$ and $\Tt(p,q')$ diffeomorphic as trisections for distinct values of $q$?
\end{question}

\begin{figure}[h!]
	\centering
	\includegraphics[width=.75\textwidth]{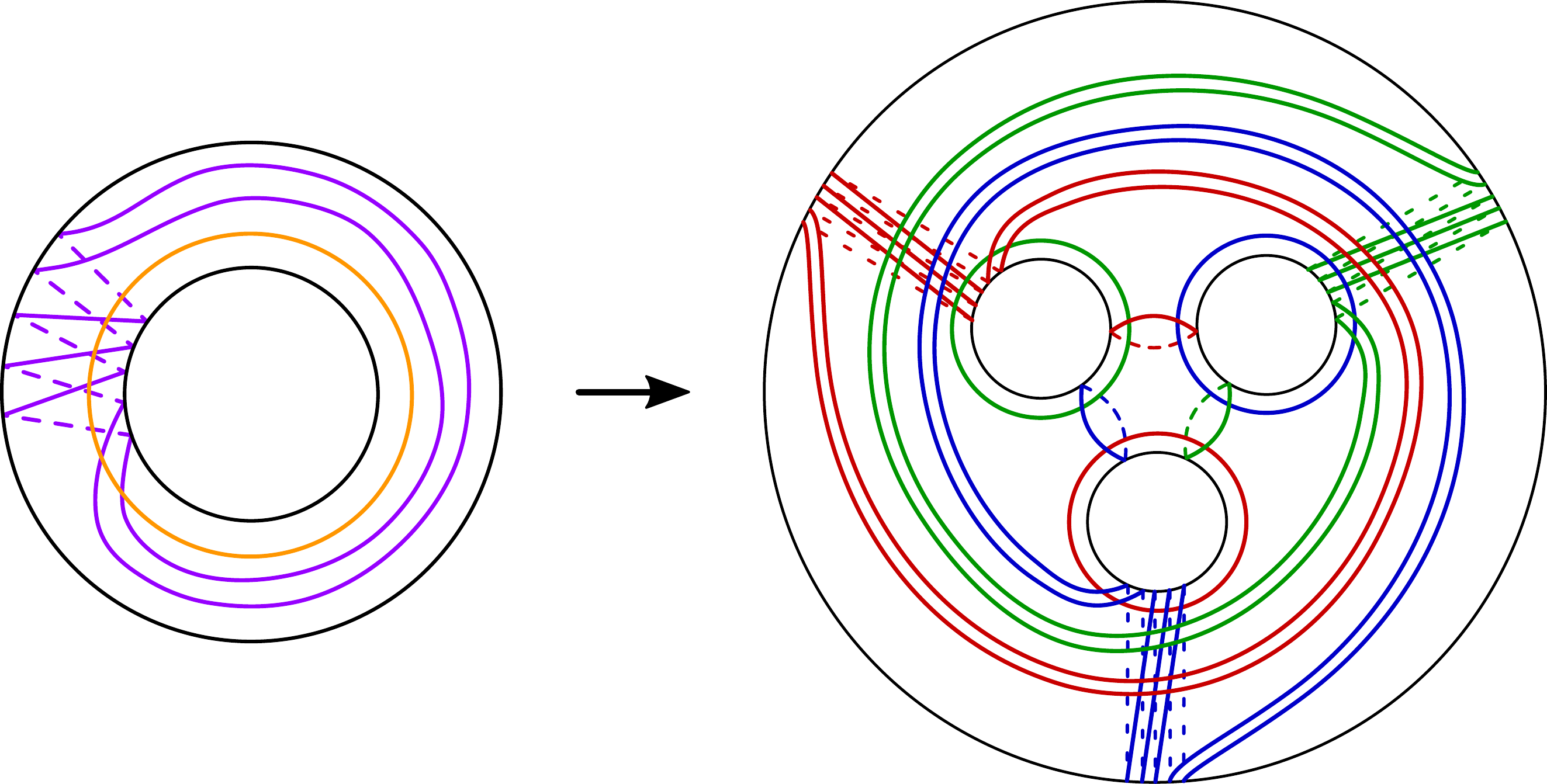}
	\caption{A genus one Heegaard diagram for the lens space $L(5,2)$ is transformed into a genus three trisection diagram for the spun lens space $\Ss_5\cong\Ss(L(5,2))$.}
	\label{fig:L52}
\end{figure}

For completeness, we describe how to obtain diagrams for the $\Ss'_p$.  Although, these diagrams depend on understanding the Gluck twist and surgery operations from a trisection diagram perspective, the details of which are the subject of work-in-progress with David Gay~\cite{Gay-Meier_Trisections_}.  The relevant sequence of diagrams is shown in Figure~\ref{fig:Surgery}.  Begin with a diagram for $\Ss_p$. (In this example, $p=4$ and the diagram comes from $\Ss(L(4,1))$.) We place points in the two central hexagons (one on the top of the surface and one on the bottom).  Colored arcs are given to show that the points can be connected in the complement of curves of each color. The fact that the arcs can be slide to coincide (paying attention to the relevant color) ensures that this is a doubly-pointed Heegaard triple. Let $\Kk$ denote the 2--knot in $\Ss_p$ encoded thusly.  We surger the surface along the dots, and extend the colored arcs to curves across the new annulus.  The resulting diagram describes the result of surgery on $\Kk$.  An easy exercise shows that this diagram destabilizes to give the genus one diagram for $S^1\times S^3$. (This proves that we identified the correct 2--knot.)  Finally, the third diagram describes the result of performing a Gluck twist on $\Kk$ in $\Ss_p$, which, by definition, gives $\Ss_p'$.  Details justifying these diagrammatic changes will appear in~\cite{Gay-Meier_Trisections_}.

\begin{remark}
	The right diagram in Figure~\ref{fig:Surgery} is obtained from the left one by a Dehn twist of one $\gamma$--curve about a $\beta$--curve.  If we had twisted the other $\gamma$--curve about the other $\beta$--curve, we would have a diagram for $\Ss^*(L(p,q))$, as described by Theorem~\ref{thm:diags}.
\end{remark}

\begin{figure}[h!]
	\centering
	\includegraphics[width=.9\textwidth]{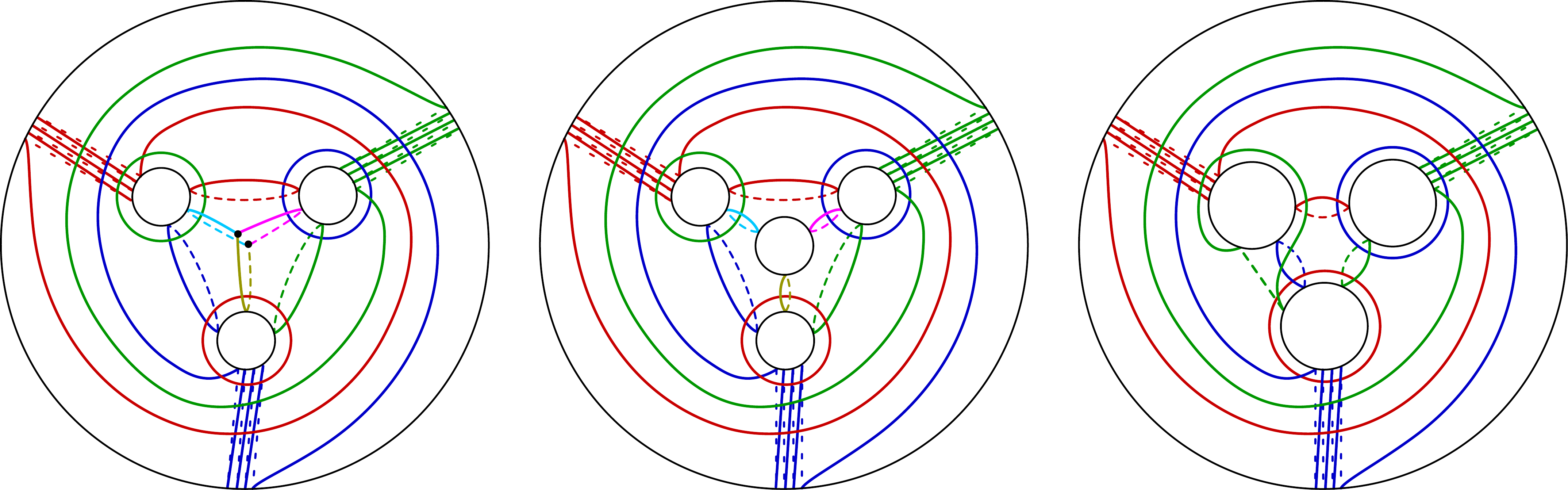}
	\caption{(Left) A doubly-pointed trisection diagram encoding the relevant 2--knot in $\Ss_p$. (Middle) The trisection diagram corresponding to the result of performing surgery on this 2--knot in $\Ss_p$.  An easy exercise shows that this diagram destabilizes to give the standard diagram for $S^1\times S^3$.  (Right) The diagram corresponding to the result of performing a Gluck twist on this 2--knot in $\Ss_p$; i.e., the sibling manifold $\Ss_p'$. (Here, $p=4$.)}
	\label{fig:Surgery}
\end{figure}

Baykur and Saeki have independently identified the manifolds in $\Pp$ as admitting genus three trisections~\cite{Baykur-Saeki_Simplifying_2017}.  In fact, they show they admit special trisections that they call \emph{simplified}.  The proof of Theorem~\ref{thm:SpunTri} gives a different type of ``simplified'' trisection for these spaces.  This leads to the the following questions.

\begin{questions}\ 
	\begin{enumerate}
		\item If $X$ admits a simplified genus three trisection (in either sense), is $X\in\Pp$?
		\item If $X$ admits a genus three trisection, does $X$ admit a simplified genus three trisection?
	\end{enumerate}
\end{questions}

%%%%%%%%%%%%%%%%%%%%%%%%%%%%%%%%%%%%%%%%%%%%%%%%%
\subsection{Spinning homology spheres}\ 
%%%%%%%%%%%%%%%%%%%%%%%%%%%%%%%%%%%%%%%%%%%%%%%%%

Let $\Sigma(p,q,r)$ denote the homology sphere that is a Seifert fibered space over the base orbifold $S^2(p,q,r)$.  Such spaces are known as \emph{Brieskhorn spheres}.  When $pq+qr+rp = \pm 1$, we can consider $\Sigma(p,q,r)$ as the branched double cover of $S^3$ along the pretzel knot $P(p,q,r)$.  In this case, it is particularly easy to give a genus two Heegaard splitting for $\Sigma(p,q,r)$ via the 3--bridge splitting of $P(p,q,r)$.  Such a diagram is shown on the left in Figure~\ref{fig:PeHS} in the case of $\Sigma(-2,3,5)$, which is the Poincar\'e homology sphere.

Figure~\ref{fig:PeHS} shows how to obtain a trisection diagram for $\Ss(\Sigma(p,q,r))$ when $pq+qr+rp=\pm 1$.  As far as we know, these are the simplest possible trisection diagrams for homology four-spheres.

\begin{figure}[h!]
	\centering
	\includegraphics[width=\textwidth]{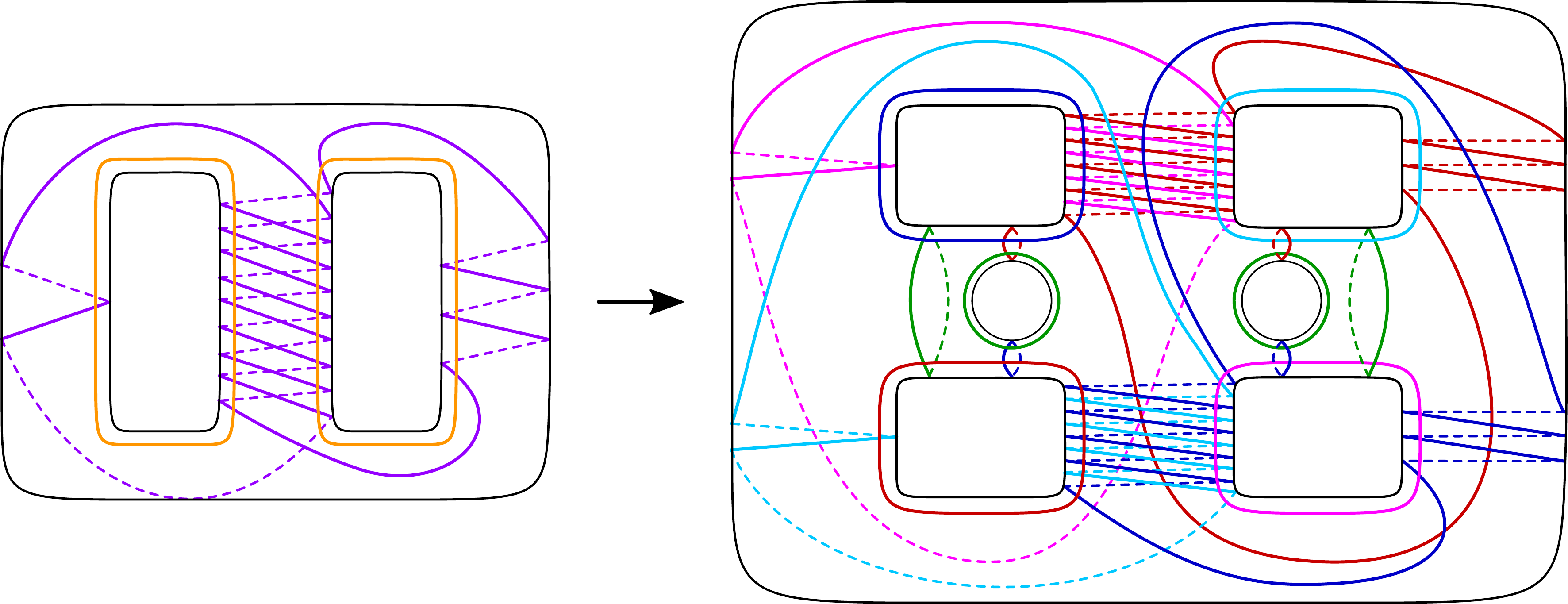}
	\caption{(Left) A Heegaard splitting for the Poincar\'e homology sphere $\Sigma(-2,3,5)$.  (Right) A trisection diagram for $\Ss(\Sigma(-2,3,5))$.  Note that two of the $\gamma$--curves (green) are not shown, but can be taken to be the same as the two complicated $\alpha$--curves (red/pink).}
	\label{fig:PeHS}
\end{figure}

%%%%%%%%%%%%%%%%%%%%%%%%%%%%%%%%%%%%%%%%%%%%%%%%%
\subsection{Spinning manifold pairs}\ 
%%%%%%%%%%%%%%%%%%%%%%%%%%%%%%%%%%%%%%%%%%%%%%%%%

We conclude by presenting two diagrams of spun pairs, one coming from a knot in $S^3$ and the other coming from a knot in a lens space. First, consider the doubly-pointed diagram for the torus knot $T(3,4)$ shown on the left in Figure~\ref{fig:TorusKnot}. One interesting property about torus knots is that the bridge number of $T(p,q)$ is equal to $\min(p,q)$.  This was used in~\cite{Meier-Zupan_Bridge_2015} to show that the spins $\Ss(T(p,q))$ have bridge number $3\min(p,q)+1$.  On the other hand, every torus knot can be isotoped to lie on the genus one Heegaard splitting of $S^3$, and, therefore, $T(p,q)$ admits a doubly-pointed genus one Heegaard diagram.  It follows, as is shown on the right side of Figure~\ref{fig:TorusKnot}, that $\Ss(T(p,q))$ admits a doubly-pointed genus three trisection diagram.

\begin{figure}[h!]
	\centering
	\includegraphics[width=.8\textwidth]{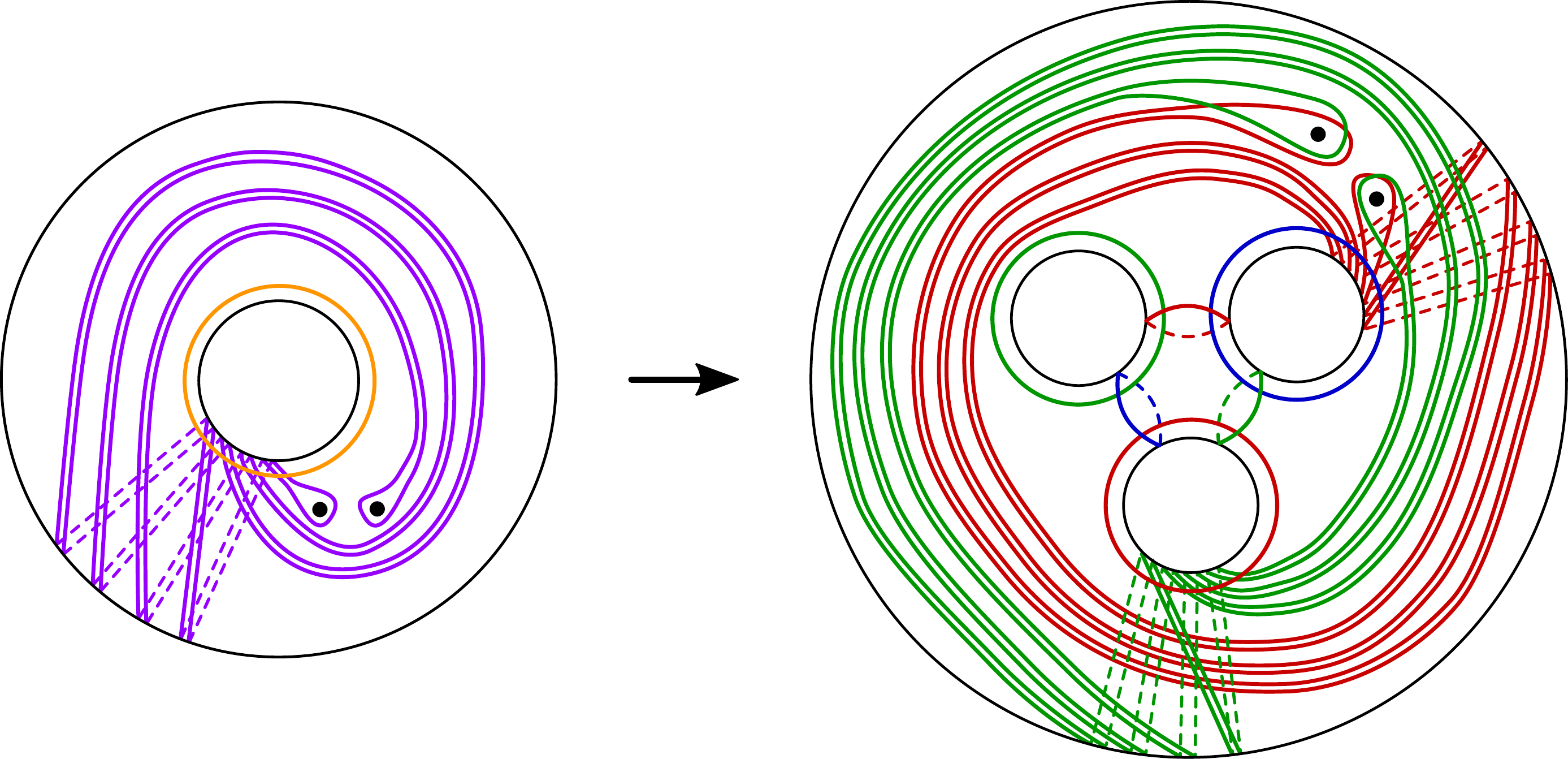}
	\caption{(Left) A doubly-pointed Heegaard splitting for the torus knot $T=T(3,4)$.  (Right) A doubly-pointed trisection diagram for the pair $\Ss(S^3,T)$.  Note that the third $\beta$--curve (blue) is not shown, but can be assumed to coincide with the complicated $\gamma$-curve (green).}
	\label{fig:TorusKnot}
\end{figure}

Next, let $Y = L(7,3)$, and let $K$ be the knot described by the doubly-pointed Heegaard diagram on the left side of Figure~\ref{fig:SimpleKnot}.  The knot $K$ is an example of a knot in $Y$ that has a surgery to $S^3$. (See~\cite{Hedden_On-Floer_2011} for an overview of these so-called \emph{simple knots}.) Figure~\ref{fig:SimpleKnot} shows the corresponding doubly-pointed trisection diagram for $\Ss(Y,K)$.

\begin{figure}[h!]
	\centering
	\includegraphics[width=.8\textwidth]{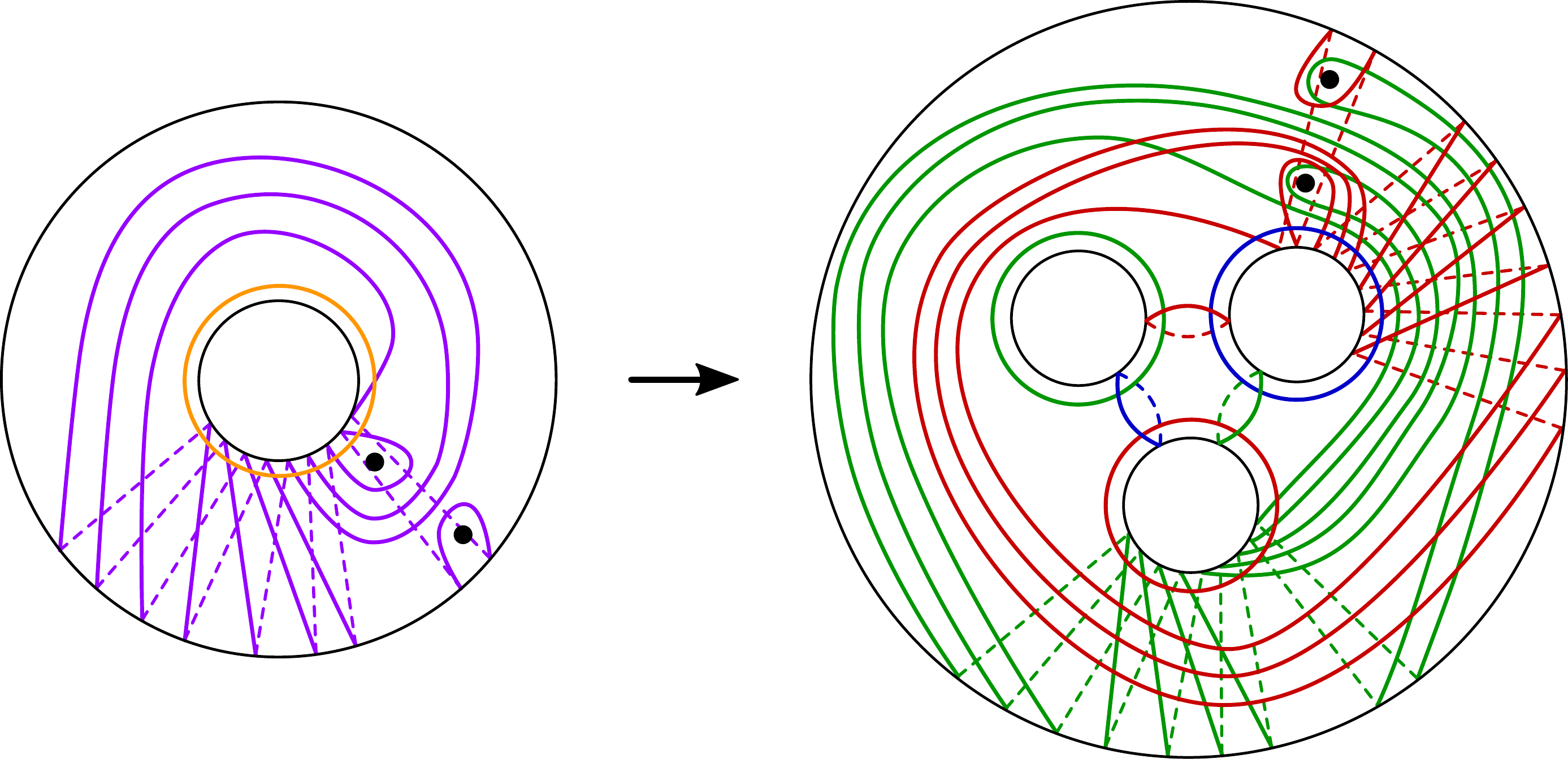}
	\caption{(Left) A doubly-pointed Heegaard splitting for a simple knot $K$ in $L(7,3)$.  (Right) A doubly-pointed trisection diagram for the pair $\Ss(L(7,3),K)$.  Note that the third $\beta$--curve (blue) is not shown, but can be assumed to coincide with the complicated $\gamma$-curve (green).}
	\label{fig:SimpleKnot}
\end{figure}

%%%%%%%%%%%%%%%%%%%%%%%%%%%%%%%%%%%%%%%%%%%%%%%%%%%%%%%%
%%%%%%%%%%%%%%%%%%%%%%%%%%%%%%%%%%%%%%%%%%%%%%%%%%%%%%%%
%%%%%%%%%%%%%%%%%%%%%%%%%%%%%%%%%%%%%%%%%%%%%%%%%%%%%%%%

\bibliographystyle{acm}
\bibliography{MasterBibliography_2017_08}

\end{document}